\documentclass[a4paper,10pt,reqno]{amsart}
\usepackage{amsmath,amsthm,amssymb,amsfonts,enumerate,color}
\usepackage{hyperref}
\usepackage{mathrsfs}
\usepackage{shapepar}

\newcommand{\R}{\mathbb{R}}

\newcommand{\tr}{\operatorname{tr}}

\newcommand{\e}{\varepsilon}

\newtheorem{theorem}{Theorem}[section]
\newtheorem{lemma}[theorem]{Lemma}
\newtheorem{corollary}[theorem]{Corollary}
\newtheorem{proposition}[theorem]{Proposition}
\newtheorem{remark}[theorem]{Remark}
\newtheorem{definition}[theorem]{Definition}

\numberwithin{equation}{section}

\author[F. Feo]{Filomena Feo}
\address{Dipartimento di Ingegneria\\
          Universit\`a degli Studi di Napoli ``Parthenope''\\
          Napoli, 80143\\
          Italy}
\email{filomena.feo@uniparthenope.it}

\author[P. R. Stinga]{Pablo Ra\'ul Stinga}
\address{Department of Mathematics\\
         Iowa State University \\
         396 Carver Hall\\
         Ames, IA 50011\\
         USA}
\email{stinga@iastate.edu}

\author[B. Volzone]{Bruno Volzone}
\address{Dipartimento di Ingegneria\\
          Universit\`a degli Studi di Napoli ``Parthenope''\\
          Napoli, 80143\\
          Italy}
\email{bruno.volzone@uniparthenope.it}



\begin{document}

\title[Fractional nonlocal Ornstein--Uhlenbeck equation]{The fractional
nonlocal Ornstein--Uhlenbeck equation, \\ Gaussian symmetrization and regularity}

\begin{abstract}
For $0<s<1$, we consider the Dirichlet problem for the fractional nonlocal Ornstein--Uhlenbeck equation
$$\begin{cases}
(-\Delta+x\cdot\nabla)^su=f,&\hbox{in}~\Omega,\\
u=0,&\hbox{on}~\partial\Omega,
\end{cases}$$
where $\Omega$ is a possibly unbounded open subset of $\R^n$, $n\geq2$.
The appropriate functional settings for this nonlocal equation and its corresponding extension problem are developed.
We apply Gaussian symmetrization techniques to derive a concentration comparison estimate for solutions.
As consequences, novel $L^p$ and $L^p(\log L)^\alpha$
regularity estimates in terms of the datum $f$ are obtained
by comparing $u$ with half-space solutions.
\end{abstract}

\maketitle

\section{Introduction}

In the present paper we are interested in developing Gaussian symmetrization techniques and, as consequences, to
obtain novel $L^p$ and $L^p(\log L)^\alpha$ regularity estimates for solutions to
nonlocal equations driven by fractional powers of the Ornstein--Uhlenbeck (OU for short) operator subject to homogeneous
Dirichlet boundary conditions.
More precisely, we focus on problems of the form
\begin{equation}\label{Problema}
\begin{cases}
(-\Delta+x\cdot\nabla)^su=f,&\hbox{in}~\Omega,\\
u=0,&\hbox{on}~\partial\Omega,
\end{cases}\qquad\hbox{for}~0<s<1,
\end{equation}
where $\Omega$ is an open subset of
$\mathbb{R}^{n}$ with $\gamma(\Omega)<1$. Here $\gamma$ denotes the Gaussian measure on $\R^{n}$, see \eqref{eq:Gaussian density}.

Our problem \eqref{Problema} corresponds to a Markov process. Indeed,
there is a stochastic process $Y_t$ having as generator the fractional OU operator
\eqref{Problema} with homogeneous Dirichlet boundary condition.
The process can be obtained as follows.
We first kill an OU process $X_t$ at $\tau_\Omega$, the first exit time of $X_t$
from the domain $\Omega$. Let us denote the killed OU process by $X_t^\Omega$.
Then we subordinate the killed OU process $X_t^\Omega$ with an $s$-stable
subordinator $T_t$. Thus $Y_t=X^\Omega_{T_t}$ is the resulting process (see for instance \cite{Applebaum}).
As explained in \cite{Caffarelli-Stinga}, \eqref{Problema} also arises in the context of
nonlinear elasticity as the Signorini problem or the thin obstacle problem.
Nonlocal equations with fractional powers of the OU operator in $\Omega=\R^n$ have been
studied in the past. Indeed, a Harnack inequality for nonnegative solutions was proved in \cite{Stinga-Zhang}.
Fractional isoperimetric problems and semilinear equations in infinite
dimensions (Wiener space) have been considered in \cite{NPS} and \cite{NPS2}.
Fractional functional inequalities were recently analyzed in \cite{Caffarelli-Sire}.

The symmetrization techniques in elliptic and parabolic PDEs are nowadays very classical and efficient tools
to derive optimal \textit{a priori} estimates for solutions.
The investigation in such direction started with the fundamental paper by
H. Weinberger \cite{Wein62}, see also \cite{Maz}. The ideas were later fully formalized by
G. Talenti in \cite{Ta1} for the homogeneous Dirichlet problem
associated to a linear equation in divergence form with zero order term
on a bounded domain of $\R^{n}$.
In particular, \cite{Ta1} establishes a strong pointwise comparison between the Schwarz spherical rearrangement
of the solution $u(x)$ to the original problem, and the unique radial solution $v(|x|)$ of a suitable elliptic problem
defined on a ball having the same measure as the original domain and radial data.
In turn, this kind of result allows to obtain regularity estimates of solutions with optimal constants.
When dealing with parabolic equations, any form of pointwise comparison between the solution $u(x,t)$ of an
initial boundary value problem and the solution $v(|x|,t)$ of a related radial problem with respect to $x$
is in general no longer available. Indeed, in this case a weaker comparison result in the integral form, the so-called mass concentration comparison (or
comparison of concentrations), holds for all times $t>0$, see for instance \cite{Band2,Mossino}.  For a detailed survey
on this theory we refer the interested reader to \cite{VAZSURV}.

Quite recently, symmetrization techniques have been successfully applied to a class of
fractional nonlocal equations. More precisely, results in terms of symmetrization were obtained
for equations driven by the fractional
Dirichlet Laplacian $$(-\Delta_D)^su=f,$$ and by the fractional Neumann
Laplacian $$(-\Delta_N)^su=f,$$ in bounded domains of $\R^n$, for $0<s<1$. These equations arise in
several important applications, see for example
\cite{Allen, Caffarelli-Stinga, Song-Vondracek, StingaVolz}. The fractional operators above are defined in terms of
the corresponding eigenfunction expansions. Then the characterization provided
by the extension problem of \cite{Stinga-Torrea}
via the Dirichlet-to-Neumann map for a (degenerate or singular) elliptic PDE
allows to treat the above-mentioned problems with local techniques
(we also refer the reader to \cite{Caffarelli-Silvestre}
for the fractional Laplacian on $\R^n$ and to \cite{Gale-Miana-Stinga} for the most general extension result
available, namely, for infinitesimal generators of integrated semigroups in Banach spaces).
This information was essential to start a program regarding the applications
of symmetrization in PDEs with fractional Laplacians. Indeed,
the first paper in such direction was the seminal work \cite{VolzDib} for the case of the fractional Dirichlet Laplacian.
Those ideas were extended and enriched with many other applications to nonlinear fractional parabolic
equations in \cite{VAZVOL3,VAZVOL1,VAZVOL2}.
When Neumann boundary conditions in fractional elliptic and parabolic problems are assumed,
the symmetrization tools applied to the extension problem
still lead to a comparison result, though of a different type, see \cite{VOLZNEUM}.

It is important to notice that all the comparison results in the nonlocal setting we just mentioned
are not pointwise in nature, but in the form of mass concentration comparison.
One motivation of such phenomenon relies on the fact that the symmetrization
argument applies on the extension problem with respect to the spatial variable $x$,
by freezing the extra extension variable $y>0$.
In other words, a comparison of the solution to the extension problem
is given in terms of the so-called Steiner symmetrization.

On the other hand, for elliptic equations involving the OU operator
$$\mathcal{L}=-\Delta+x\cdot\nabla,$$
the first comparison result through symmetrization, in the pointwise form, was obtained in \cite{Bet-Pster}.
The symmetrization has to take into account the natural variational structure of the OU operator.
Indeed, the Dirichlet problem for $\mathcal{L}$ is of the form
\begin{equation}\label{Problemalocale}
\begin{cases}
-\operatorname{div}(\varphi\nabla u)=f\varphi,&\hbox{in}~\Omega, \\
u=0,&\hbox{on}~\partial\Omega,
\end{cases}
\end{equation}
where $\varphi=\varphi(x)$ is the density of the Gaussian measure $d\gamma$ with respect to the Lebesgue measure:
\begin{equation}\label{eq:Gaussian density}
d\gamma(x)=\varphi(x)\,dx=(2\pi)^{-n/2}\exp(-|x|^2/2)\,dx,\quad\hbox{for}~x\in\mathbb{R}^{n}.
\end{equation}
The source term $f$ is then taken in the suitable class of weighted $L^p$ spaces.
Moreover, the meaningful case is when $\Omega$ is an unbounded open set. Here
we assume
$$\gamma(\Omega)<1.$$
Hence, the comparison result must be done through \emph{Gaussian symmetrization}
instead of the usual Schwarz symmetrization.
In this setting, one of the main tools in the proof is the
Gaussian isoperimetric inequality, which states that among all measurable subsets
of $\mathbb{R}^{n}$ with prescribed Gaussian measure, the half-space is the minimizer of the Gaussian perimeter.
It becomes rather intuitive to guess that the Schwarz spherical rearrangement of a function
(which is a special radial, decreasing function), appearing in the comparison results in the Lebesgue setting,
should now be replaced by the rearrangement with respect to the Gaussian measure.
The latter is a particular increasing function, depending only on one variable, defined in a half-space (see Subsection \ref{Gaussian rearrangements} for
definitions and related properties).
The authors of \cite{Bet-Pster} were able to apply this powerful machinery
to compare the solution $u$ (in the sense of rearrangement) to \eqref{Problemalocale}
with the solution $v$ to the problem
\begin{equation}\label{Problemalocalesymm}
\begin{cases}
-\operatorname{div}(\varphi\nabla v)=f^{\displaystyle\star}\varphi,&\hbox{in}~\Omega^{\displaystyle\star}\\
v=0,&\hbox{on}~\partial\Omega^{\displaystyle\star},
\end{cases}
\end{equation}
where $\Omega^{\displaystyle\star}$ is a half-space having the same Gaussian measure as $\Omega$
and $f^{\displaystyle\star}$ is the $n$-dimensional Gaussian rearrangement of $f$.
The solution $v$ to \eqref{Problemalocalesymm} (parallel to the classical case described in \cite{Ta1})
can be explicitly written, allowing the authors to derive the sharp a priori pointwise estimate
$$u^{\displaystyle\star}(x)\leq v(x),\quad\hbox{for}~x\in \Omega^{\displaystyle\star}.$$
This was the starting point to obtain regularity results for $u$ in Lorentz--Zygmund spaces.
Generalizations of this result for elliptic and parabolic problems involving
elliptic operators in divergence form which are degenerate with respect to the Gaussian measure
are contained in \cite{chiacchio,dFP}, see also references therein.

Our main concern is to get sharp estimates
for the solution $u$ to \eqref{Problema} by comparing it with the solution $\psi$ to the problem
\begin{equation}\label{problema simm}
\begin{cases}
\mathcal{L}^{s}\psi=f^{\displaystyle\star}, & \hbox{in}~\Omega^{\displaystyle\star},\\
\psi=0, & \hbox{on}~\partial\Omega^{\displaystyle\star}.
\end{cases}
\end{equation}
As our previous discussion evidences, \eqref{problema simm} is actually a one dimensional problem.
Our idea that yields the desired result reads as follows.
Using the main extension result of \cite{Stinga-Torrea} we can characterize the fractional OU operator
$\mathcal{L}^s$ in \eqref{Problema} as a suitable Dirichlet-to-Neumann map.
This allows us to obtain the solution $u$ to \eqref{Problema}
as the trace on $\Omega$ of the solution $w=w(x,y)$ of the following degenerate elliptic boundary value problem, which will be called the \emph{extension problem}
associated to \eqref{Problema}:
\begin{equation}\label{Problema estensione}
\begin{cases}
-\operatorname{div}(y^a\varphi(x)\nabla_{x,y}w)=0, &\hbox{in}~\mathcal{C}_{\Omega},\\
w=0,&\hbox{on}~\partial_{L}\mathcal{C}_{\Omega},\\
-\underset{y\rightarrow0^{+}}{\lim}y^aw_y=f,&\hbox{on}~\Omega.
\end{cases}
\end{equation}
Here
\begin{equation}\label{eq:a}
a:=1-2s\in(-1,1),
\end{equation}
while
\[
\mathcal{C}_{\Omega}:=\Omega\times(0,\infty)
\]
is the infinite cylinder of basis $\Omega$, and
$\partial_{L}\mathcal{C}_{\Omega}:=\partial\Omega\times[0,\infty)$ is its lateral boundary.
In a similar way, the solution $\psi$ to \eqref{problema simm} can be seen as
the trace over $\Omega^{\displaystyle\star}$ of the solution $v=v(x,y)$ to
\begin{equation}\label{Problema estensione sim}
\begin{cases}
-\operatorname{div}(y^a\varphi(x)\nabla_{x,y}v)=0, &\hbox{in}~\mathcal{C}_{\Omega}^{\displaystyle\star},\\
v=0,&\hbox{on}~\partial_{L}\mathcal{C}_{\Omega}^{\displaystyle\star},\\
-\underset{y\rightarrow0^{+}}{\lim}y^av_y=f^{\displaystyle\star},&\hbox{on}~\Omega^{\displaystyle\star},
\end{cases}
\end{equation}
where
\begin{equation}\label{eq:Steiner symmetrization}
\mathcal{C}_{\Omega}^{\displaystyle\star}:=\Omega^{\displaystyle\star}\times(0,\infty),
\end{equation}
and $\partial_{L}\mathcal{C}_{\Omega
}^{^{\displaystyle\star}}:=\partial\Omega^{\displaystyle\star}\times[0,\infty)$.
Therefore, the problem reduces to look for a \emph{mass concentration
comparison} between the solution $w$ to \eqref{Problema estensione} and the solution $v$ to
\eqref{Problema estensione sim}. More precisely, we prove that
\begin{equation}
\int_{0}^rw^{\circledast}(\sigma,y)\,d\sigma\leq \int_{0}^rv^{\circledast}(\sigma,y)\,d\sigma,
\quad\hbox{for all}~r\in[0,\gamma(\Omega)],\label{concineq}
\end{equation}
where, for all $y\geq0$, the functions $w^{\circledast}(\cdot,y)$ and
$v^{\circledast}(\cdot,y)$ are the one dimensional Gaussian rearrangements
of $w(\cdot,y)$ and $v(\cdot,y)$, respectively. The key role of this framework
is played by a novel second order derivation formula
for functions defined by integrals, see Corollary \ref{Secondderivform}, whose proof presents new nontrivial technical difficulties owed to the Gaussian
framework.
As a consequence, we will obtain $L^p$ and $L^p(\log L)^\alpha$ estimates for $u$ in terms of $f$.

The paper is organized as follows. Section \ref{Preliminari} contains
the preliminaries needed for the developments of our results.
In particular, we briefly describe some basic properties of the Gaussian measure and the
OU semigroup. Moreover, we carefully develop a full and self-contained analysis of the main
functional setting where problems \eqref{Problema} and \eqref{Problema estensione} are posed.
Section \ref{Preliminari} ends with the introduction of the basic definitions
and properties of symmetrization with respect to the Gaussian measure. In this regard, we will present the proof of the derivation formula stated in
Theorem \ref{Firstderivform}, whose consequence is the above-mentioned second order differentiation formula, see Corollary \ref{Secondderivform}.
Section \ref{Section:comparison} is entirely devoted to the proof
of the comparison \eqref{concineq}, that is, our main result Theorem \ref{primoteoremadiconfronto}.
In Section \ref{Section:regularity} we present our novel
Gaussian--Zygmund $L^p(\log L)^\alpha(\Omega,\gamma)$ and $L^p(\Omega,\gamma)$
regularity estimates for solutions $u$ in
terms of the datum $f$, see Theorem \ref{thm:integrability}.
More precisely, our main result (Theorem \ref{primoteoremadiconfronto})
is combined with $L^p(\log L)^\alpha$ regularity estimates of the solution $\psi$
to problem \eqref{problema simm}, which is obtained by using the explicit form
of $\psi$ in terms of the fractional integral $\mathcal{L}^{-s}(f^{\star})$ and the OU semigroup.
Finally, in the Appendix we shall use suitable estimates of the Mehler kernel
to exhibit a semigroup-based proof of the regularity estimates when the datum
$f$ belongs to the smaller Gaussian--Lebesgue space $L^p(\Omega,\gamma)$.

\section{Preliminaries, functional setting, and the second order derivation formula}\label{Preliminari}

In this section we recall the basic tools we
are going to use in the proof of our main comparison result,
Theorem \ref{primoteoremadiconfronto}, and its
consequences. First, we introduce some basics about Gaussian analysis and the OU semigroup.
Then the necessary functional
background to precise the fractional nonlocal equations \eqref{Problema} and \eqref{problema simm},
and their extension problems \eqref{Problema estensione} and \eqref{Problema estensione sim}
will be developed. Finally, after presenting
definitions and properties of rearrangement techniques in the Gaussian
framework, we will prove our novel
second order derivation formula, see Theorem \ref{Firstderivform}
and Corollary \ref{Secondderivform}.

\subsection{Gaussian analysis and the OU semigroup}

\subsubsection{Gaussian measure and isoperimetry}

Let $d\gamma$ be the $n$-dimensional normalized Gaussian measure on $\mathbb{R}^{n}$
defined in \eqref{eq:Gaussian density}. Let $\Omega$ be an open subset of $\R^n$, possibly unbounded.
We denote by $H^{1}(\Omega,\gamma)$ the Sobolev space
with respect to the Gaussian measure,
which is obtained as the completion of $C^{\infty}(\overline{\Omega})$ with respect to the norm
$$\|u\|_{H^{1}(\Omega,\gamma)}^2=\int_{\Omega}
u^{2}\,d\gamma(x)+\int_{\Omega}|\nabla u|^{2}\,d\gamma(x).$$
By $H^{1}_{0}(\Omega,\gamma)$ we denote
the closure of $C_c^{\infty}(\Omega)$ in the norm of $H^{1}(\Omega,\gamma)$.
The following Poincar\'{e} inequality holds (see for instance \cite{Eh2}):
if $\gamma(\Omega)<1$ then there exists a constant $C_\Omega>0$ such that
\begin{equation}
\int_{\Omega}| u|^{2}\,d\gamma(x)\leq C_{\Omega}\int_{\Omega}
|\nabla u|^{2}\,d\gamma(x),\quad\hbox{for all}~u\in H_{0}^{1}(\Omega,\gamma).
\label{Poincare}
\end{equation}
One of the main tools to prove the comparison result is the \emph{Gaussian
isoperimetric inequality}. Let us define the
perimeter with respect to Gaussian measure as
$$P(E)=\int_{\partial E}\varphi(x)\,d\mathcal{H}^{n-1}(x),$$
where $E$ is a set of locally finite perimeter and $\partial E$ denotes its reduced boundary. As usual, $\mathcal{H}^{n-1}$ denotes the
$(n-1)$-dimensional Hausdorff measure. It is well known (see \cite{Bo})
that among all measurable sets of $\mathbb{R}^{n}$ with prescribed Gaussian
measure, the half-spaces take the smallest perimeter. More precisely, we have
\begin{equation}
P(E)\geq\frac{1}{\sqrt{2\pi}}\exp\big(-[\Phi^{-1}(\gamma(E))]^2/2\big),
\label{dis isop}
\end{equation}
for all subsets $E\subset\mathbb{R}^{n}$, where, for $\lambda\in\mathbb{R}\cup\{-\infty,+\infty\}$, we set
\begin{equation}\label{Phi}
\Phi(\lambda):=\frac{1}{\sqrt{2\pi}}\int_{\lambda}^{\infty}e^{-r^2/2}\,dr.
\end{equation}

\subsubsection{The OU semigroup}

We recall some remarkable properties of the OU semigroup (see \cite{AS,bok} for further details)
which will turn out to be useful in the following.

The solution to the Cauchy problem
$$\begin{cases}
\rho_{t}+\mathcal{L}\rho=0,&\hbox{in}~\mathbb{R}^{n}\times(0,\infty), \\
\rho(x,0)=g(x),&\hbox{on}~\mathbb{R}^{n},
\end{cases}$$
is given by the OU semigroup
$$\rho(x,t)=e^{-t\mathcal{L}}g(x).$$
It is a classical fact that such a semigroup can be expressed in terms of a suitable integral kernel.
More precisely, if $g\in L^{p}(\mathbb{R}^{n},\gamma)$, for $1\leq p\leq\infty$, then
\begin{equation}\label{solutsemig}
e^{-t\mathcal{L}}g(x)=\int_{\mathbb{R}^{n}}M_{t}(x,y)g(y)\,d\gamma(y),\quad\hbox{for}~x\in\R^n,~t>0.
\end{equation}
Here $M_{t}(x,y)$ is the so-called Mehler kernel, which is defined by
\begin{equation}\label{eq:Mehler kernel}
M_{t}(x,y)=\frac{1}{(1-e^{-2t})^{n/2}}\exp\bigg(-\frac{e^{-2t}|x|^{2}-2e^{-t}\langle x,y
\rangle+e^{-2t}|y|^{2}}{2(1-e^{-2t})}\bigg).
\end{equation}
We recall that
\begin{equation}\label{M=1}
\int_{\R^n}M_{t}(x,y)\,d\gamma(y)=1,\quad\hbox{for all}~x\in\R^n,~t>0,
\end{equation}
and that if $g\in L^p(\R^n,\gamma)$, $1\leq p<\infty$, then
\begin{equation}\label{Stima Lp semigruppo}
\|e^{-t\mathcal{L}}g\|_{L^p(\R^n,\gamma)}
=\bigg\|\int_{\R^n}M_{t}(\cdot,y)g(y)\,d\gamma(y)\bigg\|_{L^{p}(\R^n,\gamma)}
\leq\left\Vert g\right\Vert _{L^{p}(\R^n,\gamma)}.
\end{equation}

It is standard to define the OU semigroup on a domain $\Omega$ of $\R^{n}$
subject to homogenous Dirichlet boundary conditions. Indeed, the solution to
the Cauchy--Dirichlet problem
\begin{equation} \label{p2}
\begin{cases}
\eta_{t}+\mathcal{L}\eta=0,&\hbox{in}~\Omega\times(0,\infty), \\
\eta(x,t)=0,&\hbox{on}~\partial\Omega\times[0,\infty), \\
\eta(x,0)=f(x),&\hbox{on}~\Omega,
\end{cases}
\end{equation}
is given by the semigroup generated by the OU in $\Omega$ with Dirichlet boundary conditions:
$$\eta(x,t)=e^{-t\mathcal{L}_{\Omega}}f(x).$$
It follows from standard parabolic regularity theory that $\eta$ is smooth in $\Omega\times(0,\infty)$.
Now, let us choose $\Omega=H$, where $H$ is the half-space $H:=\{x=(x_{1},x')\in\R^n :x_{1}>0,~x'\in\R^{n-1}\}$
and define
\begin{equation}\label{est1}
\widetilde{f}(x)=
\begin{cases}
f(x_{1},x^{\prime}),&\hbox{for}~x\in H,\\
-f(-x_{1},x^{\prime}),&\hbox{for}~x\in\R^n\setminus H.
\end{cases}
\end{equation}
Observe that for $1\leq p<\infty$ we have
\begin{equation}\label{norma prolungamento}
\|\widetilde{f}\|_{L^p(\R^n,\gamma)}=2\|f\|_{L^{p}(H,\gamma)}.
\end{equation}
It is not difficult to check (see for example \cite{Priola}) that in this case
the semigroup associated to \eqref{p2} is obtained as the restriction to $H$
of the OU semigroup on $\R^n$ applied to $\tilde{f}$, that is,
\begin{equation}\label{eta}
\eta(x,t)=e^{-t\mathcal{L}_H}f(x)=e^{-t\mathcal{L}}\widetilde{f}(x)\big|_H.
\end{equation}
Moreover, using the expression of the OU semigroup in terms of the Mehler kernel \eqref{solutsemig}
we see that the following explicit formula holds in dimension $n=1$:
\begin{equation}\label{eq:eta}
\eta(x,t)=\int_{0}^{\infty}\left[M_{t}(x,y)-M_{t}(x,-y)\right]f(y)\,d\gamma(y),
\quad\hbox{for all}~x>0,~t>0.
\end{equation}

\subsection{The fractional nonlocal OU equation and the extension problem}

We introduce now an appropriate functional setting, which is essential when dealing with problems \eqref{Problema} and
\eqref{Problema estensione}. In order to define the fractional powers
$\mathcal{L}^su$, $0<s<1$, we consider the sequence of eigenvalues
$0<\lambda_1\leq \lambda_2\leq\cdots\leq\lambda_k\nearrow\infty$
and the corresponding
orthonormal basis of Dirichlet eigenfunctions $\{\psi_{k}\}_{k\geq1}$ of
$\mathcal{L}$ in $L^{2}(\Omega,\gamma)$, see for example \cite{Betta-Chiacchio-Ferone}.
In other words, for every $k\geq1$, $\psi_{k}\in L^2(\Omega,\gamma)$
is a weak solution to the Dirichlet problem
$$\begin{cases}
-\operatorname{div}(\varphi\nabla\psi_{k})=\lambda_{k}\varphi\,\psi_{k},&\hbox{in}~\Omega,\\
\psi_{k}=0,&\hbox{on}~\partial\Omega.
\end{cases}$$
Now, let us define the Hilbert space
\[
\mathcal{H}^s(\Omega,\gamma)\equiv\operatorname{Dom}(\mathcal{L}^{s}
):=\Big\{u\in L^{2}(\Omega,\gamma):\sum_{k=1}^{\infty}\lambda_{k}
^s|\langle u,\psi_{k}\rangle_{L^{2}(\Omega,\gamma)}|^{2}<\infty\Big\},
\]
with scalar product
\[
\langle u,v\rangle_{\mathcal{H}^s(\Omega,\gamma)}:=\sum_{k=1}^{\infty}
\lambda_{k}^{s}\langle u,\psi_{k}\rangle_{L^{2}(\Omega,\gamma)}\langle
v,\psi_{k}\rangle_{L^{2}(\Omega,\gamma)}.
\]
Then the norm in $\mathcal{H}^s(\Omega,\gamma)$ is given by
\[
\Vert u\Vert_{\mathcal{H}^s(\Omega,\gamma)}^{2}=\sum_{k=1}^{\infty}\lambda
_{k}^{s}|\langle u,\psi_{k}\rangle_{L^{2}(\Omega,\gamma)}|^{2}.
\]
For $u\in\mathcal{H}^s(\Omega,\gamma)$, we define $\mathcal{L}^{s}u$ as the
element in the dual space $\big(\mathcal{H}^s(\Omega,\gamma)\big)^{\prime}$ through the formula
\[
\mathcal{L}^{s}u=\sum_{k=1}^{\infty}\lambda_{k}^{s}\langle u,\psi
_{k}\rangle_{L^{2}(\Omega,\gamma)}\psi_{k},\quad\hbox{in}~\big(\mathcal{H}^s(\Omega,\gamma)\big)^{\prime}.
\]
That is, for any function $v\in\mathcal{H}^s(\Omega,\gamma)$ we have
\[
\langle\mathcal{L}^su,v\rangle=\sum_{k=1}^{\infty}\lambda_{k}^{s}\langle
u,\psi_{k}\rangle_{L^{2}(\Omega,\gamma)}\langle v,\psi_{k}\rangle
_{L^{2}(\Omega,\gamma)}=\langle u,v\rangle_{\mathcal{H}^s(\Omega,\gamma)}.
\]
This identity can be rewritten as
\[
\langle\mathcal{L}^{s}u,v\rangle=\int_{\Omega}(\mathcal{L}^{s/2}
u)(\mathcal{L}^{s/2}v)\,dx,\quad\hbox{for every}~u,v\in\mathcal{H}^s(\Omega,\gamma),
\]
where $\mathcal{L}^{s/2}$ is defined by taking the power $s/2$ of the
eigenvalues $\lambda_{k}$.

\begin{remark}[The fractional OU operator is a nonlocal operator]
By using the method of semigroups as in \cite{Stinga-Torrea}, see also
\cite{Caffarelli-Stinga, StingaVolz, Stinga-Zhang}, it can be seen that the fractional
operator $\mathcal{L}^s$ is a nonlocal operator. Indeed, we have the semigroup and kernel formulas
\begin{align*}
\mathcal{L}^su(x) &=\frac{1}{\Gamma(-s)}\int_0^\infty\big(e^{-t\mathcal{L}_\Omega}u(x)
-u(x)\big)\,\frac{dt}{t^{1+s}} \\
&=\operatorname{PV}\int_{\Omega}\big(u(x)-u(y)\big)K_s(x,y)\,dy+u(x)B_s(x),
\end{align*}
where $\operatorname{PV}$ means that the integral is taken in the principal value sense. Here
$$e^{-t\mathcal{L}_\Omega}u(x)=\int_\Omega H_t(x,y)u(y)\,d\gamma(y),$$
is the semigroup generated by $\mathcal{L}$ in $\Omega$ with Dirichlet boundary conditions,
$H_t(x,y)$ is the corresponding heat kernel,
$$K_s(x,y)=\frac{1}{|\Gamma(-s)|}\int_0^\infty H_t(x,y)\,\frac{dt}{t^{1+s}},\quad x,y\in\Omega,$$
and
$$B_s(x)=\frac{1}{|\Gamma(-s)|}\int_0^\infty\big(1-e^{-t\mathcal{L}_\Omega}1(x)\big)\,\frac{dt}{t^{1+s}},
\quad x\in\Omega.$$
In the particular case of $\Omega=\R^n$, we have $H_t(x,y)=M_t(x,y)$, the Mehler kernel, and,
as a direct consequence of \eqref{M=1}, we see that $B_s(x)\equiv0$. Though this description
is important, we will not use it here. Instead, we will apply the extension technique.
\end{remark}

Recalling the notation in \eqref{eq:a}, we define the Sobolev energy space on the infinite cylinder $\mathcal{C}_{\Omega}$:
\[
H_{0,L}^{1}(\mathcal{C}_{\Omega},d\gamma(x)\otimes y^ady)=\bigg\{  v\in
H_{\mathrm{loc}}^{1}(\mathcal{C}_{\Omega}):v=0~\hbox{on}~\partial_{L}
\mathcal{C}_{\Omega},~\iint_{\mathcal{C}_{\Omega}}y^a(v^{2}+|\nabla
_{x,y}v|^{2})\,d\gamma(x)\,dy<\infty\bigg\}.
\]
By the Gaussian Poincar\'{e} inequality \eqref{Poincare}, for each
$v\in H_{0,L}^{1}(\mathcal{C}_{\Omega},d\gamma(x)\otimes y^ady)$ we have
\begin{align*}
\iint_{\mathcal{C}_{\Omega}}y^av^{2}\,d\gamma(x)\,dy  &  =\int_{0}^{\infty
}y^a\int_{\Omega}v^{2}\,d\gamma(x)\,dy\leq C_{\Omega}\int_{0}^{\infty}
y^a\int_{\Omega}|\nabla_{x}v|^{2}d\gamma(x)\,dy\\
&  \leq C_{\Omega}\iint_{\mathcal{C}_{\Omega}}y^a|\nabla_{x,y}v|^{2}
\,d\gamma(x)\,dy.
\end{align*}
Thus we can equip the space $H_{0,L}^{1}(\mathcal{C}_{\Omega},d\gamma
(x)\otimes y^ady)$ with the equivalent norm
\[
\Vert v\Vert_{H_{0,L}^{1}(\mathcal{C}_{\Omega},d\gamma(x)\otimes y^ady)}
^{2}=\iint_{\mathcal{C}_{\Omega}}y^a|\nabla_{x,y}v|^{2}\,d\gamma(x)\,dy,
\]
which is actually the norm defined through the scalar product
\[
\langle v,w\rangle_{H_{0,L}^{1}(\mathcal{C}_{\Omega},d\gamma
(x)\otimes y^ady)}=\iint_{\mathcal{C}_{\Omega}}y^a\,\nabla_{x,y}v\cdot \nabla_{x,y}w\,\,d\gamma(x)\,dy.
\]
Furthermore, since we can identify $H_{0,L}^{1}(\mathcal{C}_{\Omega},d\gamma(x)\otimes y^ady)$
with the space $H^{1}((0,\infty),y^ady;H_{0}^{1}(\Omega,\gamma))$,
we have that $H_{0,L}^{1}(\mathcal{C}_{\Omega},d\gamma(x)\otimes y^ady)$ is a
Hilbert space.\\[0.5pt]

The following Theorem is a particular case of \cite[Theorem~1.1]{Stinga-Torrea},
see also \cite{Caffarelli-Stinga, Gale-Miana-Stinga, Stinga-Zhang}.
It provides the characterization of $\mathcal{L}^su$ as the Dirichlet-to-Neumann map
for a degenerate elliptic extension problem in the upper cylinder $\mathcal{C}_{\Omega}$, for any
$u\in\mathcal{H}^s(\Omega,\gamma)$. As the solution $w(x,y)$
is explicitly given by \eqref{trueextens} and \eqref{StingaTorreasemigr},
the proof is just a verification of the statements, see for example \cite{Stinga-Torrea,StingaVolz}.

\begin{theorem}[Extension problem]\label{extensth}
Let $u\in\mathcal{H}^s(\Omega,\gamma)$. Define
\begin{equation}
w(x,y)\equiv\mathcal{P}_y^su(x)=
\frac{2^{1-s}}{\Gamma(s)}\sum_{k=1}^{\infty}(\lambda_k^{1/2}y)^s\mathcal{K}_s(\lambda_k^{1/2}y)
\langle u,\psi_{k}\rangle_{L^{2}(\Omega,\gamma)}\psi_{k}(x),
\label{trueextens}
\end{equation}
for $y\geq0$, where $\mathcal{K}_s$ is the modified Bessel function of the second kind and order $0<s<1$.
Then $w\in H_{0,L}^{1}(\mathcal{C}_{\Omega},d\gamma(x)\otimes y^ady)$ and it is the
unique weak solution to the extension problem
\begin{equation}
\begin{cases}
-\operatorname{div}_{x,y}(y^a\varphi(x)\nabla_{x,y}w)=0,&\hbox{in}~\mathcal{C}_{\Omega},\\
w=0, & \hbox{on}~\partial_{L}\mathcal{C}_{\Omega},\\
w(x,0)=u(x), & \hbox{on}~\Omega,
\end{cases}
\label{extensionproblem}%
\end{equation}
that vanishes weakly as $y\to\infty$. More precisely,
\[
\iint_{\mathcal{C}_{\Omega}}y^a(\nabla_{x,y}w\cdot\nabla_{x,y}\xi)\,d\gamma
(x)\,dy=0,
\]
for all test functions $\xi\in H_{0,L}^{1}(\mathcal{C}_{\Omega},d\gamma
(x)\otimes y^ady)$ with zero trace over $\Omega$, $\operatorname{tr}_{\Omega}%
\xi=0$, and
\[
\lim_{y\rightarrow0^{+}}w(x,y)=u(x)
\]
in $L^{2}(\Omega,\gamma)$. Furthermore, the function $w$ is the unique
minimizer of the energy functional
\begin{equation}
\mathcal{F}(v)=\frac{1}{2}\iint_{\mathcal{C}_{\Omega}}y^a|\nabla_{x,y}%
v|^{2}\,d\gamma(x)\,dy, \label{energufunct}%
\end{equation}
over the set $\mathcal{U}=\left\{  v\in H_{0,L}^{1}(\mathcal{C}_{\Omega
},d\gamma(x)\otimes y^ady):\,\operatorname{tr}_{\Omega}v=u\right\}  $.
We can also write
\begin{equation}
w(x,y)=\frac{y^{2s}}{4^s\Gamma(s)}\int_{0}^{\infty}e^{-y^{2}/(4t)}
e^{-t\mathcal{L}_{\Omega}}u(x)\,\frac{dt}{t^{1+s}}.\label{StingaTorreasemigr}
\end{equation}
Moreover,
$$-\lim_{y\rightarrow0^{+}}y^aw_{y}=c_s\mathcal{L}^su,\quad\hbox{in}~\big(\mathcal{H}^s
(\Omega,\gamma)\big)^{\prime},$$
where $c_s=\frac{\Gamma(1-s)}{4^{s-1/2}\Gamma(s)}>0$. Finally, the following energy identity holds:
\begin{equation}\label{importidentity}
\iint_{\mathcal{C}_{\Omega}}y^a|\nabla_{x,y}w|^{2}\,d\gamma(x)\,dy=
c_s\|\mathcal{L}^{s/2}u\|^{2}_{L^{2}(\Omega,\gamma)}.
\end{equation}
\end{theorem}

Theorem \ref{extensth} shows in particular that the domain
$\mathcal{H}^s(\Omega,\gamma)$ is contained in the range of the trace operator
on $H_{0,L}^{1}(\mathcal{C}_{\Omega},d\gamma(x)\otimes y^ady)$ at $y=0$. The next Lemma
shows that actually these two spaces coincide.

\begin{lemma}[Trace inequality]\label{Identification}
We have
\[
\operatorname{tr}_{\Omega}(H_{0,L}^{1}(\mathcal{C}_{\Omega},d\gamma(x)\otimes
y^ady))=\mathcal{H}^s(\Omega,\gamma).
\]
Moreover, for all $v\in H^{1}_{0,L}(\mathcal{C}_{\Omega},d\gamma(x)\otimes y^ady)$,
\begin{equation}\label{trace}
\|\mathcal{L}^{s/2}v(x,0)\|_{L^{2}(\Omega,\varphi)}^2 \leq (2c_s)^{-1}
\iint_{\mathcal{C}_{\Omega}}
y^a|\nabla_{x,y}v|^{2}\,d\gamma(x)\,dy.
\end{equation}
In particular, equality holds in \eqref{trace} if $v=\mathcal{P}_y^s(\operatorname{tr}_{\Omega}v)(x)$,
(see \eqref{trueextens}).
\end{lemma}

\begin{proof}
Let $u=\tr_\Omega v$, for $v\in H_{0,L}^{1}(\mathcal{C}_{\Omega
},d\gamma(x)\otimes y^ady)$ and define the function $w$ as in \eqref{trueextens}.
It is readily checked that $w$ satisfies \eqref{extensionproblem}, so it minimizes the functional $\mathcal{F}$ in \eqref{energufunct}. Therefore, by \eqref{importidentity},
$\|\mathcal{L}^{s/2}u\|^{2}_{L^{2}(\Omega,\gamma)}\leq (c_s)^{-1}\Vert
v\Vert_{H_{0,L}^{1}(\mathcal{C}_{\Omega},d\gamma(x)\otimes y^ady)}^{2}<\infty$,
that is, $u\in\mathcal{H}^s(\Omega,\gamma)$. Now \eqref{trace} is clear.
\end{proof}

\begin{proposition}[Compactness of the trace embedding]\label{compactness}We have
\[
\operatorname{tr}_{\Omega}(H_{0,L}^{1}(\mathcal{C}_{\Omega},d\gamma(x)\otimes
y^ady))\subset\subset L^{2}(\Omega,\gamma).
\]
\end{proposition}

\begin{proof}
We need to check that the trace operator $\operatorname{tr}_\Omega:
H_{0,L}^{1}(\mathcal{C}_{\Omega},d\gamma(x)\otimes y^ady)\to L^{2}(\Omega,\gamma)$ is compact.
It is clear that $\operatorname{tr}_\Omega$ is continuous from
$H_{0,L}^{1}(\mathcal{C}_{\Omega},d\gamma(x)\otimes y^ady)$ into $L^2(\Omega,\gamma)$
since \eqref{trace} holds. Similarly, the finite rank operators $T_j$, $j\geq1$, defined by
$$T_jv=\sum_{k=1}^j\langle v(\cdot,0),\psi_k\rangle_{L^2(\Omega,\gamma)}\psi_k,$$
are continuous from
$H_{0,L}^{1}(\mathcal{C}_{\Omega},d\gamma(x)\otimes y^ady)$ into $L^{2}(\Omega,\gamma)$.
By using \eqref{trace} and the fact that $\lambda_k\nearrow\infty$, as $k\to\infty$, we see that,
if $v\in H_{0,L}^{1}(\mathcal{C}_{\Omega},d\gamma(x)\otimes y^ady)$,
\begin{align*}
\|T_jv-\tr_\Omega v\|_{L^2(\Omega,\gamma)}^2 &= \sum_{k=j+1}^\infty|\langle v(\cdot,0),\psi_k\rangle|^2 \\
&\leq \frac{1}{\lambda_{j+1}^s}\sum_{k=j+1}^\infty\lambda_k^s|\langle v(\cdot,0),\psi_k\rangle|^2
\leq \frac{1}{\lambda_{j+1}^s}\|v\|^2_{H_{0,L}^{1}(\mathcal{C}_{\Omega},d\gamma(x)\otimes y^ady)}.
\end{align*}
Therefore $T_j$ converges to $\tr_\Omega$ in the operator norm, as $j\to\infty$, and $\tr_\Omega$ is compact.
\end{proof}

Using the previous preliminaries, it is natural to give the following definitions of weak solutions.

\begin{definition}[Weak solution of \eqref{Problema estensione}]\label{Defprobest}
Let $f\in L^{2}(\Omega,\gamma)$. We say that $w\in H^{1}_{0,L}
(\mathcal{C}_{\Omega},d\gamma(x)\otimes y^ady)$ is a weak solution to the linear Dirichlet-Neumann extension problem \eqref{Problema estensione} if
\begin{equation}\iint_{\mathcal{C}_{\Omega}}y^a\nabla_{x,y}w\cdot\nabla_{x,y}v\,d\gamma
(x)\,dy=c_s^{-1}\int_{\Omega}f(x)v(x,0)\,d\gamma(x),\label{weakform}\end{equation}
for every $v\in H^{1}_{0,L}(\mathcal{C}_{\Omega},d\gamma(x)\otimes y^ady)$, where
$c_s>0$ is the constant appearing in Theorem \ref{extensth}.
\end{definition}

\begin{definition}[Weak solution of \eqref{Problema}]
If $w$ is the weak solution to \eqref{Problema estensione}, its trace $u:=w(\cdot,0)\in \mathcal{H}^{s}(\Omega,\gamma)$ on $\Omega$ will be called a weak
solution to
\eqref{Problema}.

\end{definition}
\begin{remark}
If we assume that $f$ is in the dual space $\mathcal{H}^{s}(\Omega,\gamma)^{\prime}$, it is clear that the right hand side
in \eqref{weakform} must be replaced by the dual product
$\langle f,v(\cdot,0)\rangle$. Then the (unique) solution $u$ to \eqref{Problema} will be again the trace over $\Omega$ of the unique solution $w$ to the
extension problem \eqref{Problema estensione}.
\end{remark}

The following is just a restatement of Theorem \ref{extensth}, see \cite[Theorem~1.1]{Stinga-Torrea} and
also \cite{Gale-Miana-Stinga}.

\begin{theorem}[Extension problem for negative powers]
\label{existweaksollin}
Given $f\in L^{2}(\Omega,\gamma)$, let $u\in\mathcal{H}^s(\Omega,\gamma)$
be the unique solution to problem \eqref{Problema}.
The solution $w$ (see \eqref{trueextens})
to the extension problem \eqref{extensionproblem} can be written as
\begin{equation}\label{eq:semigroup formula f}
\begin{aligned}
w(x,y) &= \frac{2^{1-s}}{\Gamma(s)}\sum_{k=1}^{\infty}(\lambda_k^{1/2}y)^s\mathcal{K}_s(\lambda_k^{1/2}y)
\frac{\langle f,\psi_{k}\rangle_{L^{2}(\Omega,\gamma)}}{\lambda_k^s}\psi_{k}(x) \\
&= \frac{1}{\Gamma(s)}\int_{0}^{\infty}e^{-y^{2}/(4t)}
e^{-t\mathcal{L}_{\Omega}}f(x)\,\frac{dt}{t^{1-s}}.
\end{aligned}
\end{equation}
In particular, this is the unique weak solution to \eqref{Problema estensione} and
\begin{equation}\label{eq:fractional integral}
w(x,0)=u(x)=\mathcal{L}^{-s}f(x)=\frac{1}{\Gamma(s)}\int_0^\infty e^{-t\mathcal{L}_\Omega}f(x)\,\frac{dt}{t^{1-s}}.
\end{equation}
\end{theorem}

The domain $\mathcal{H}^s(\Omega,\gamma)$ of the fractional nonlocal operator $\mathcal{L}^s$
can be characterized as a suitable interpolation space between two
Hilbert spaces. Indeed, using the abstract discrete version of the $J$-Theorem
(see for example the Appendix in \cite{SirBonfVaz}), it is straightforward to prove that
\begin{equation}
\mathcal{H}^s(\Omega,\gamma)=\left[H_{0}^{1}(\Omega,\gamma),L^{2}%
(\Omega,\gamma)\right] _{1-s},\label{Gaussianintsp}
\end{equation}
where the space in the right hand side of \eqref{Gaussianintsp}
is the real interpolation space between
$H_{0}^{1}(\Omega,\gamma)$ and $L^{2}(\Omega,\gamma)$. Then
$\mathcal{H}^{1/2}(\Omega,\gamma)$ may be seen as the equivalent of the
Lions--Magenes space $H_{00}^{1/2}(\Omega)$ in the Gaussian setting.

\subsection{Gaussian rearrangements}\label{Gaussian rearrangements}

We give the notion of rearrangement with respect to the Gaussian measure. For extra details,
we refer the interested reader to the classical monographs \cite{Bennett} and \cite{CR}.
If $u$ is a measurable function in $\Omega$, we denote by
\begin{itemize}
\item $u^{\circledast}$ the one dimensional decreasing rearrangement of $u$ with respect
to the Gaussian measure (also called \emph{one dimensional Gaussian rearrangement of $u$}):
$$u^{\circledast}(r)=\inf\{t\geq0:\gamma_{u}(t)\leq r\},\quad r\in(0,\gamma(\Omega)],$$
where $\gamma_{u}(t)=\gamma(\{x\in\Omega:\left\vert u(x)\right\vert>t\})$ is the distribution function of $u$;
\item $u^{\displaystyle\star}$ the $n$-dimensional rearrangement of $u$ with respect the Gaussian measure:
$$u^{\displaystyle\star}(x)=u^{\circledast}\big(\Phi(x_{1})\big),\quad x\in\Omega^{\displaystyle\star},$$
where $\Omega^{\displaystyle\star}=\{x=(  x_{1},\ldots,x_{n})\in\mathbb{R}^{n}:x_{1}>\lambda\}$ is the
half-space such that $\gamma(\Omega^{\displaystyle\star})=\gamma(\Omega)$
and $\Phi$ is given by \eqref{Phi}.
\end{itemize}
By definition, $u^{\displaystyle\star}$ is a function which depends
only on the first variable $x_{1}$, it is increasing and its level sets are half-spaces. Moreover,
$u$, $u^{\circledast}$ and $u^{\displaystyle\star}$ have the same
distribution function. This implies that the Gaussian $L^{p}$ norm is invariant under these rearrangements:
$$\|u\|_{L^{p}(\Omega,\gamma)}=\|u^{\circledast}\|_{L^{p}(0,\gamma(\Omega))}=\|u^{{\displaystyle\star}}\|_{L^{p}(\Omega^{{\displaystyle\star}},\gamma)},\quad\hbox{for
any}~1\leq p\leq\infty.$$
If $u$ is defined on a half-space and $u=u^{\displaystyle\star}$ we sometimes say that $u$ is \emph{rearranged}. Furthermore,
if $u$ and $v$ are
measurable functions then the following Hardy-Littlewood inequality holds:
\begin{equation}
\int_{\Omega}|u(x)v(x)|\,d\gamma(x)\leq\int_{\Omega^{\displaystyle\star}}
u^{\displaystyle\star}(x)v^{\displaystyle\star}(x)\,d\gamma(x)
=\int_{0}^{\gamma(\Omega)}u^{\circledast}(r)v^{\circledast}(r)\,dr.
\label{Hardy-Litt}%
\end{equation}
If $u$ is defined on $\Omega$, $v$ on $\Omega^{\displaystyle\star}$ and
the following estimate holds
\begin{equation}
\int_{0}^{\gamma(\Omega)}u^{\circledast}(r)\,dr\leq \int_{0}^{\gamma(\Omega)}v^{\circledast}(r)\,dr,\label{massconcent}
\end{equation}
the same inequality is called \emph{mass concentration inequality} (or \emph{comparison of mass concentration}). If $v=v^{\displaystyle\star}$ and
\eqref{massconcent} occurs, we also say that $u^{\displaystyle\star}$ is \emph{less concentrated} that $v$ and we write $u^{\displaystyle\star}\prec v$.
Moreover, \eqref{massconcent} implies that (see for instance \cite{Chong})
\[
\|u\|_{L^{p}(\Omega,\gamma)}\leq \|v\|_{L^{p}(\Omega^{\displaystyle\star},\gamma)},
\quad\hbox{for all}~1\leq p\leq\infty.
\]
We will often deal with two-variable functions
\begin{equation}\label{w}%
w:(x,y)\in\mathcal{C}_{\Omega}=\Omega\times(0,\infty)\rightarrow w(x,y)\in{\mathbb{R}},
\end{equation}
which are measurable with respect to $x$.
In such a case it will be convenient to consider the
so-called \textit{Gaussian Steiner symmetrization} of $\mathcal{C}_{\Omega}$ with
respect to the variable $x$, namely, the set $\mathcal{C}_{\Omega}^{\displaystyle\star}$
as defined in \eqref{eq:Steiner symmetrization}.
In addition (see for instance \cite{chiacchio,Eh2}) we will denote by $\gamma_{w}(t,y)$
and $w^{\circledast}(r,y)$ the distribution
function and the one dimensional Gaussian decreasing rearrangements of \eqref{w}, with respect to $x$,
for each $y$ fixed. We will also define the function
$$w^{\displaystyle\star}(x,y)=w^{\circledast}\big(\Phi(x_{1}),y\big),$$
which is called the \emph{Gaussian Steiner symmetrization of $w$}, with respect to $x$, that is, with respect to the line $x=0$.
Clearly, for any fixed $y$, $w^{\displaystyle\star}(\cdot,y)$ is an increasing function depending only on $x_{1}$.

Now we recall a result that we will use in the proof of our main comparison result in Section
\ref{Section:comparison}.

\begin{proposition}[See {\cite[p.~255]{chiacchio}}]\label{prop chiacchio}
Consider the Cauchy--Dirichlet problem \eqref{p2} with $\Omega=\Omega^{\displaystyle\star}$.
If $f(x)=f^{\displaystyle\star}(x)$ for a.e. $x \in \Omega^{\displaystyle\star}$ and
$f^{\displaystyle\star}\in L^{2}(\Omega^{\displaystyle\star},\gamma)$, then the solution $\eta$ to
\eqref{p2} is such that $\eta(x,t)=\eta^{\displaystyle\star}(x,t)$, for a.e. $x\in\Omega^{\displaystyle\star}$
and for all $t\geq0$.
\end{proposition}

\subsection{The second order derivation formula}

It will be essential for us to be able to differentiate
with respect to the extra variable $y$ under the integral symbol in the expression
\[\int_{\{x:\,w(x,y)>w^{\circledast}(r,y)\}}\frac{\partial w}{\partial y}(x,y)\,d\gamma(x).\]
Equivalently, we need to derive the Gaussian version of the first and second order differentiation
formulas established for the Lebesgue measure in \cite{Al-Li-TRo, Band2, Ferone-Mercaldo, Mossino}. The first order differentiation formula can be stated
as follows:

\begin{proposition}[See \cite{chiacchio}, also \cite{simon}]
If $w\in H^{1}(0,T;L^{2}(\Omega,\gamma))$ is a nonnegative function,
for some $T>0$, then $w^{\circledast}\in H^{1}\big(0,T;L^{2}(0,\gamma(\Omega))\big)$.
In addition, if $\gamma(\{w(x,t)=w^{\circledast}(r,t)\})=0$ for a.e.
$(r,t)\in(0,\gamma(\Omega))\times(0,T)$, then the following derivation formula holds%
\begin{equation}\label{Rakotoson derivation formula}
\int_{\left\{x:\,w(x,y)>w^{\circledast}(r,y)\right\}}
\frac{\partial}{\partial y}w(x,y)\,d\gamma(x)=\int_{0}^{r}
w^{\circledast}(\sigma,y)\,d\sigma.
\end{equation}
\end{proposition}

In order to prove our novel second order derivation formula,
we need the following version of the coarea formula (see \cite{Federer} and \cite[Theorem 11]{Haj}).

\begin{proposition}
\label{Prop coarea copy(1)}If $u\in W_{\mathrm{loc}}^{1,p}(\R^n)$, with $p>1$,
and $\psi:\R^n\to\R$ is a nonnegative measurable function, then there exists a representative of
$u$, denoted again by $u$, such that
\begin{equation}\label{coarea}%
\int_{\R^n}\psi(x)|\nabla_{x}u|\,dx=\int_{-\infty}^{\infty}\bigg(\int_{\left\{x:\,u(x)=\tau\right\}}\psi(x)\,d\mathcal{H}^{n-1}(x)\bigg)\,d\tau.
\end{equation}
\end{proposition}

Now we present our new Gaussian derivation formulas, which are a nonstandard adaptation of the
derivation formula exhibited in \cite{Ferone-Mercaldo}.

\begin{theorem}\label{Firstderivform}
Let $0<\varepsilon<T<\infty$. Consider a nonnegative function
$$w=w(x,y)\in H_{0,L}^{1}(\mathcal{C}_{\Omega},d\gamma(x)\otimes y^ady)\cap C^{1}(\Omega\times(\varepsilon,T)).$$
Suppose also that $w$ is $C^{1,\alpha}$ with respect to $y\in(\varepsilon,T)$,
for some $0<\alpha\leq1$, uniformly with respect to $x\in \Omega$.
Moreover, assume that $f(x,y)$ is a continuous function
on the cylinder $\mathcal{C}_{\Omega}$ such that $f\in H^{1}\normalcolor(\mathcal{C}_{\Omega},d\gamma(x)\otimes y^a \,dy)$
and the function $f(x,y)\varphi(x)$
is Lipschitz with respect to $y\in(  \varepsilon,T)$, uniformly with respect to $x\in$ $\Omega$.
Furthermore, suppose that
\begin{equation}
\gamma\Big(\big\{x\in\Omega:|\nabla_{x}w|=0,w(x,y)\in(0,\sup_{x}w(x,y))\big\}\Big)
=0,\quad\hbox{for all}~y\in(\varepsilon,T),
\label{cond aggiuntiva}%
\end{equation}
and set
\[
H(t,y):=\int_{\left\{x:\,w(x,y)>t\right\}}f(x,y)\,d\gamma(x),
\]
for $t\in[0,\infty)$ and $y\in(\varepsilon,T)$. The following statements hold true.
\begin{enumerate}[$(i)$]
\item For any fixed $y\in(\varepsilon,T)$, $H(t,y)$ is differentiable with respect to $t$ for a.e. $t\geq0$ and
\begin{equation}
\frac{\partial}{\partial t}H(t,y)=-\int_{\left\{x:\,w(x,y)=t\right\}}
\frac{f(x,y)}{|\nabla_{x}w|}\varphi(x)\,d\mathcal{H}^{n-1}(x).
\label{t derivate}%
\end{equation}
\item For any fixed $t\geq0$, $H(t,y)$ is differentiable with respect to $y$ and,
for a.e $y\in(\varepsilon,T)$,
\begin{equation}
\frac{\partial}{\partial y}H(t,y)=\displaystyle\int_{\left\{x:\,w(x,y)>t\right\}}
\frac{\partial}{\partial y}f(x,y)\,d\gamma(x)+
\displaystyle\int_{\left\{x:\,w(x,y)=t\right\}}\frac{\partial}{\partial y}w(x,y)\frac{f(x,y)}{|\nabla_{x}w|}\varphi(x)\,d\mathcal{H}^{n-1}(x).
\label{y derivate}%
\end{equation}
\end{enumerate}
\end{theorem}

\begin{proof}
Let us first prove $(i)$. By the extension theorem (see for instance \cite{FEOPOST})
we can extend $w(\cdot,y)$ as a function in $H^1(\R^n)$, for a.e. $y>0$.
Condition \eqref{cond aggiuntiva} allows us to choose
$\psi(x)=\frac{f(x,y)}{|\nabla_xw|}\varphi(x)\chi_{\{w(x,y)>t\}}(x)$ and $u(x)=w(x,y)$
in the coarea formula \eqref{coarea} to get
$$\int_{\left\{x:\,w(x,y)>t\right\}}f(x,y)\,d\gamma(x)=\int_{t}^{\infty}\bigg(\int_{\left\{x:\,w(x,y)=\tau\right\}}
\frac{f(x,y)}{|\nabla_{x}w|}\varphi(x)\,d\mathcal{H}^{n-1}(x)\bigg)\,d\tau,$$
for a.e. $t\geq0$. Thus \eqref{t derivate} follows.

Next we prove $(ii)$.
We observe that
\[
H(t,y)-H(t,\overline{y})=\triangle_{1}+\triangle_{2}+\triangle_{3},
\]
where
$$\triangle_{1}=\int_{\left\{x:\,w(x,\overline{y})>t\right\}}
[f(x,y)-f(x,\overline{y})]\,d\gamma(x),\quad
\triangle_{2}=\int_{\left\{x:\,w(x,y)>t\geq w(x,\overline{y})\right\}}f(x,y)\,d\gamma(x),$$
and
$$\triangle_{3}=-\int_{\left\{x:\,w(x,\overline{y})>t\geq w(x,y)\right\}}f(x,y)\,d\gamma(x).$$
Since $f(x,y)\varphi(x)$ is Lipschitz with respect
to $y$, uniformly in $x$, by Lebesgue's dominated convergence theorem we easily infer that
\begin{equation}
\underset{y\rightarrow\overline{y}}{\lim}\frac{\triangle_1}{y-\overline{y}}=\int_{\left\{x:\,w(x,\overline{y})>t\right\}}
\frac{\partial f}{\partial y}(x,\overline{y})\,d\gamma(x),
\label{DELTA 1}%
\end{equation}
for a.e. $t$ and a.e. $\overline{y}\in(\varepsilon,T)$.
Let us next consider $\frac{\triangle_{2}}{y-\overline{y}}.$ We have%
\begin{equation}
\frac{\triangle_{2}}{y-\overline{y}}=\frac{1}{y-\overline{y}}\displaystyle\int_{D_{1}}f(x,y)\,d\gamma(x)
+\frac{1}{y-\overline{y}}\int_{D_{2}}f(x,y)\,d\gamma(x),
\label{divisa}%
\end{equation}
where
\[
D_{1}=\left\{x\in\Omega:w(x,y)>t\geq w(x,\overline{y}),\frac{\partial w
}{\partial y}(x,\overline{y})=0\right\},
\]
and
\[
D_{2}=\left\{x\in\Omega:w(x,y)>t\geq w(x,\overline{y}),\frac{\partial w
}{\partial y}(x,\overline{y})\neq0\right\}.
\]
We claim that
\begin{equation}
\lim_{y\rightarrow\overline{y}}\frac{1}{y-\overline{y}}\int_{D_{1}}f(x,y)\,d\gamma(x)=0,\quad
\hbox{for a.e.}~t\geq0.
\label{zero}%
\end{equation}
Since $w(x,y)\in C^{1,\alpha}$ with respect to $y\in(\varepsilon,T)$, uniformly in $x\in\Omega$, we have
\begin{equation}
\left\vert \frac{\partial w}{\partial y}(x,y)-\frac{\partial w}{\partial
y}(x,\overline{y})\right\vert \leq c|y-\overline{y}|^{\alpha},\quad\hbox{for every}~x\in\Omega,
\label{holder continuity}%
\end{equation}
for a constant $c>0$ independent on $x$, $y$ and $\overline{y}$.
Since for any
$x\in D_{1}$ we have $\frac{\partial}{\partial y}w(x,\overline{y})=0$, by \eqref{holder continuity} we easily find the uniform estimate
\[
\left\vert w(x,y)-w(x,\overline{y})\right\vert\leq
\int_{\overline{y}}^{y}\left|\frac{\partial}{\partial
z}w(x,z)\right|dz \leq c|y-\overline{y}|^{\alpha+1},\quad\hbox{for all}~x\in D_{1},
\]
which yields
\begin{equation}\label{1quotratio}
\left\vert \frac{1}{y-\overline{y}}\int_{D_{1}}f(x,y)\,d\gamma(x)\right\vert
\leq\frac{1}{\left\vert
y-\overline{y}\right\vert }\int_{\{x:\,t-c|y-\overline{y}|^{\alpha+1}<w(x,\overline{y})\leq t\}}
|f(x,y)|\,d\gamma(x).
\end{equation}
Let us set
\[
\Psi(t):=\int_{\left\{x:\,w(x,\overline{y})>t\right\}}\max_{y\in[\varepsilon,T]}|f(x,y)|\,d\gamma(x).
\]
Since $f\in L^{2}(\mathcal{C}_{\Omega},d\gamma(x)\otimes y^ady)$
and $f$ is continuous, by Fubini's theorem we have that
$$\int_{\Omega}|f(x,y)|\,d\gamma(x)<\infty,$$
for a.e. $y>0$, and $\Psi(t)<\infty$, for all $t\geq0$. Then \eqref{1quotratio} implies
\[
\left\vert \frac{1}{y-\overline{y}}\int_{D_{1}}f(x,y)\,d\gamma(x)\right\vert
\leq c\left\vert y-\overline{y}\right\vert ^{\alpha}\frac{\Psi\big(t-c\left\vert y-\overline
{y}\right\vert ^{\alpha+1}\big)-\Psi(t)}{c\left\vert y-\overline{y}\right\vert ^{\alpha+1}}.
\]
Since the function $\Psi$ it monotone, it is also differentiable almost everywhere and then \eqref{zero} holds.
Now let us evaluate the second term in \eqref{divisa}.
First we consider the case $y>\overline{y}.$
For $y$ sufficiently close to $\overline{y},$ we have
$$\frac{1}{y-\overline{y}}\int_{D_{2}}f(x,y)\,d\gamma(x)=\frac{1}{y-\overline{y}}\int_{D_{3}}f(x,y)\,d\gamma(x),$$
where $$D_{3}=\bigg\{x\in\Omega:w(x,y)>t\geq w(x,\overline{y}),\frac
{\partial w}{\partial y}(x,\overline{y})>0\bigg\}.$$
Let us set
\[
\Gamma_{t}=\left\{x\in\Omega:w(x,\overline{y})=t\right\}  \cap\left\{
x\in\Omega:\frac{\partial w}{\partial y}(x,\overline{y})>0\right\}  .
\]
In a neighborhood $B_{r}(\overline{x},\overline{y},t)$ of a point
$(\overline{x},\overline{y},t)\in\R^{n+2}$ with $\overline{x}\in\Gamma_{t}$, the equality $w(x,y)=t$
implicitly defines a function $y=v(x,t)$ such that $\overline{y}%
=v(\overline{x},t)$ and $w(x,v(x,t))=t.$ Moreover for $y$ sufficiently close
to $\overline{y}$ we have
\[
D_{3}\cap B_{r}(\overline{x},\overline{y},t)=\left\{  x\in B_{r}(\overline
{x},\overline{y},t):\overline{y}<v(x,t)<y\right\}  .
\]
Observe that the implicit function theorem gives $|\nabla_{x}v(x,t)|=|\nabla_{x}w(x,\overline{y})|/\frac{\partial w}{\partial y}
(x,\overline{y})$. Then using the coarea formula (\ref{coarea})
we have
\begin{equation}\label{delta 2}
\begin{aligned}
\underset{y\rightarrow\overline{y}^{+}}{\lim}\frac{1}{y-\overline{y}}\int_{D_{3}\cap B_{r}(\overline{x},\overline{y},t)}
&f(x,y)\,d\gamma(x)
=\underset{y\rightarrow\overline{y}^{+}}{\lim}\frac{1}{y-\overline{y}}\int_{\overline{y}}^{y}\int_{\left\{x:\,v(x,t)=s\right\}}
\frac{f(x,y)\,\varphi(x)}{|\nabla_{x}v|}\,d\mathcal{H}^{n-1}(x)\,ds\\
&=\int_{\{x\in B_{r}(\overline{x},\overline{y},t):v(x,t)=\overline{y}\}}\frac{f(x,\overline{y})}{\left\vert \nabla
_{x}v\right\vert }\varphi(x)\,d\mathcal{H}^{n-1}(x)\\
&=\int_{\{x\in B_{r}(\overline{x},\overline{y},t):w(x,\overline{y})=t\}}\frac{\partial w}{\partial y}
(x,\overline{y})\frac{f(x,\overline{y})}{|\nabla_{x}w|}\varphi(x)\,d\mathcal{H}^{n-1}(x).
\end{aligned}
\end{equation}
By \eqref{zero} and \eqref{delta 2} it follows that
\begin{equation}
\underset{y\rightarrow\overline{y}^{+}}{\lim}\frac{\triangle_{2}}
{y-\overline{y}}=\int_{\{x:\,w(x,\overline{y})=t,\frac{\partial w}{\partial
y}(x,\overline{y})>0\}}\frac{\partial}{\partial y}w(x,\overline
{y})\frac{f(x,\overline{y})}{|\nabla_{x}w|}\varphi(x)\,d\mathcal{H}^{n-1}(x).
\label{i}%
\end{equation}
By analogous arguments we obtain
\begin{equation}
\underset{y\rightarrow\overline{y}^{-}}{\lim}\frac{\triangle_{2}}%
{y-\overline{y}}=\int_{\{x:\,w(x,\overline{y})=t,\frac{\partial w}{\partial
y}(x,\overline{y})<0\}}\frac{\partial w}{\partial y}(x,\overline
{y})\,\frac{f(x,\overline{y})}{\left\vert \nabla_{x}w\right\vert }\varphi(x)\,d\mathcal{H}^{n-1}(x).
\label{ii}%
\end{equation}
In the same way we can prove the analogue of (\ref{i})\ and (\ref{ii}) with
$\triangle_{2}\ $replaced by $\triangle_{3}.$ Then%
\begin{equation}
\underset{y\rightarrow\overline{y}}{\lim}\frac{\triangle_{2}+\triangle_{3}%
}{y-\overline{y}}=\int_{\left\{  x:\,w(x,\overline{y})=t\right\}  }%
\frac{\partial w}{\partial y}(x,\overline{y})\,\frac{f(x,\overline{y}%
)}{\left\vert \nabla_{x}w\right\vert }\varphi(x)\,d\mathcal{H}^{n-1}(x).
\label{DELTA 2 e3}%
\end{equation}
Putting together \eqref{DELTA 1} and \eqref{DELTA 2 e3} we obtain assertion $(ii)$.
\end{proof}
By recalling that the rearrangement $w^{\circledast}$ of a function $w$ is the
generalized inverse function of the distribution function $\gamma_{w}$, and
applying the chain rule formula, we can prove our novel derivation formula.

\begin{corollary}[Gaussian second order derivation formula]\label{Secondderivform}
Under the assumptions of Theorem \ref{Firstderivform}, for a.e. $y\in(\varepsilon,T)$ the following
derivation formula holds:
\begin{multline}
\frac{\partial}{\partial y}\int_{\left\{x:\,w(x,y)>w^{\circledast}(r,y)\right\}}f(x,y)\,d\gamma(x)=
\int_{\left\{x:\,w(x,y)>w^{\circledast}(r,y)\right\}}\frac{\partial}{\partial y}f(x,y)\,d\gamma(x)\\
-\int_{\left\{x:\,w(x,y)=w^{\circledast}(r,y)\right\}}\frac{f(x,y)}{|\nabla_{x}w|}
\left[\frac{{\displaystyle\int_{\left\{x:\,w(x,y)=w^{\circledast}(r,y)\right\}}}
\frac{\frac{\partial}{\partial y}w(x,y)}{|\nabla_{x}w|}\varphi(x)\,d\mathcal{H}^{n-1}(x)}
{{\displaystyle\int_{\left\{x:\,w(x,y)=w^{\circledast}(r,y)\right\}}}\frac{1}{|\nabla_{x}w|}\varphi(x)\,d\mathcal{H}^{n-1}(x)}
-\frac{\partial}{\partial y}w(x,y)\right]\varphi(x)d\mathcal{H}^{n-1}(x).\label{Der1}
\end{multline}
In particular, if $w(x,y)$ is $C^{1}$ and the functions $w(x,y)\varphi(x)$,
$\frac{\partial}{\partial y}w(x,y)\varphi(x)$ are Lipschitz in $y\in(\varepsilon,T)$,
uniformly with respect to $x\in\Omega$, we have
\begin{equation}
\begin{aligned}%
&\int_{\left\{x:\,w(x,y)>w^{\circledast}(r,y)\right\}}\frac{\partial^{2}}{\partial y^{2}}w(x,y)\,d\gamma(x) \\
&=\frac{\partial^{2}}{\partial y^{2}}
\displaystyle\int_{0}^{r}
w^{\circledast}(\sigma,y)\,d\sigma-\int_{\left\{x:\,w(x,y)=w^{\circledast}(r,y)\right\}}
\frac{\big(\frac{\partial}{\partial y}w(x,y)\big)^{2}}{|\nabla_{x}w|}\varphi(x)\,d\mathcal{H}^{n-1}(x)\\
&\quad+\Bigg(\int_{\left\{x:\,w(x,y)=w^{\circledast}(r,y)\right\}}
\frac{\frac{\partial}{\partial y}w(x,y)}{|\nabla_{x}w|}\varphi(x)d\mathcal{H}^{n-1}(x)\Bigg)^{2}
\Bigg(\int_{\left\{x:\,w(x,y)=w^{\circledast}(r,y)\right\}}
\frac{\varphi(x)}{|\nabla_{x}w|}d\mathcal{H}^{n-1}(x)\Bigg)^{-1}.\label{Der2}
\end{aligned}
\end{equation}
\end{corollary}

\begin{proof}
In order to prove \eqref{Der1} we need to evaluate the $y$-derivative of
$H(t,y)$ when $t=w^{\circledast}(r,y).$ By a rearrangement property (see for example \cite{Bennett}) we have
\begin{equation}
\int_{\left\{x:\,w(x,y)>w^{\circledast}(r,y)\right\}}w(x,y)\,d\gamma(x)=\int_{0}^{s}w^{\circledast}(\sigma,y)\,d\sigma.\label{ww}%
\end{equation}
Observe that by applying \eqref{coarea} it is not difficult to prove that
\begin{equation}
-\frac{\partial w^{\circledast}}{\partial r}=\Bigg(\int_{\left\{x:\,w(x,y)=w^{\circledast}(r,y)\right\}}\frac{\varphi(x)}{|\nabla_{x}w|}\,d\mathcal{H}^{n-1}(x)\Bigg)^{-1}\label{derivatrearrang}.
\end{equation}
Now using \eqref{coarea}, \eqref{derivatrearrang}, \eqref{Rakotoson derivation formula} and the chain rule,
\begin{align}
\frac{\partial}{\partial y}w^{\circledast}(r,y)&=\frac{\partial}{\partial y}\left(\frac{\partial}{\partial
r}\int_{0}^{r}w^{\circledast}(\tau,y)d\tau\right)\nonumber\\
&=\frac{\partial}{\partial r}\left(\frac{\partial}{\partial y}\int_{0}^{r}w^{\circledast}(\tau,y)d\tau\right)
=\frac{\partial}{\partial r}\int_{\left\{x:\,w(x,y)>w^{\circledast}(r,y)\right\}}\frac{\partial w}{\partial y}d\gamma(x)\label{tt}\\
&=\frac{\partial}{\partial r}\int_{w^{\circledast}(r,y)}^{\infty}d\tau\int_{\left\{x:\,w(x,y)=\tau\right\}}\frac{\frac{\partial }{\partial
y}w(x,y)}{|\nabla_{x}w|}\varphi(x)\,d\mathcal{H}^{n-1}(x)\nonumber\\
&=-\frac{\partial w^{\circledast}}{\partial r}\,\int_{\left\{x:\,w(x,y)=w^{\circledast}(r,y)\right\}}\frac{\frac{\partial }{\partial
y}w(x,y)}{|\nabla_{x}w|}\varphi(x)\,d\mathcal{H}^{n-1}(x)\nonumber\\
&=\frac{\displaystyle\int_{\left\{x:\,w(x,y)=w^{\circledast}(r,y)\right\}}\frac{\frac{\partial}{\partial
y}w(x,y)}{\left\vert \nabla_{x}w\right\vert }\varphi(x)\,d
\mathcal{H}^{n-1}(x)}{\displaystyle\int_{\left\{x:\,w(x,y)=w^{\circledast}(r,y)\right\}}\frac{1}{\left\vert
\nabla_{x}w\right\vert }\varphi(x)\,d\mathcal{H}^{n-1}(x)\nonumber%
}.
\end{align}
By (\ref{tt}), (\ref{t derivate}) and (\ref{y derivate}) we obtain%
\begin{align}
&\frac{\partial}{\partial y}H(w^{\circledast}(r,y),y)   =\left.
\frac{\partial}{\partial t}H(t,y)\right\vert _{t=w^{\circledast}(r,y)}%
\frac{\partial}{\partial y}w^{\circledast}(r,y)+H_{y}(w^{\circledast}(r,y),y)\label{111}\\
& =-\int_{\left\{x:\,w(x,y)=w^{\circledast}(r,y)\right\}}\frac{f(x,y)}{\left\vert \nabla
_{x}w\right\vert }\varphi(x)\,d\mathcal{H}^{n-1}(x)\times\nonumber
\frac{\displaystyle\int_{\left\{x:\,w(x,y)=w^{\circledast}(r,y)\right\}}\frac{\frac{\partial}{\partial
y}w(x,y)}{\left\vert \nabla_{x}w\right\vert }\varphi(x)\,d\mathcal{H}^{n-1}(x)}{\displaystyle\int_{\left\{x:\,w(x,y)=w^{\circledast}(r,y)\right\}}\frac{1}{\left\vert
\nabla_{x}w\right\vert }\varphi(x)\,d\mathcal{H}^{n-1}%
(x)}\nonumber\\
&\quad+ \int_{\left\{x:\,w(x,y)>w^{\circledast}(r,y)\right\}}\frac{\partial}{\partial y}f(x,y)\,d\gamma
(x)+\int_{\left\{x:\,w(x,y)=w^{\circledast}(r,y)\right\}}\frac{\partial}{\partial y}%
w(x,y)\frac{f(x,y)}{\left\vert \nabla_{x}w\right\vert }\varphi(x)\,d\mathcal{H}^{n-1}(x),\nonumber
\end{align}
which is \eqref{Der1}.
Now we are in position to prove \eqref{Der2}. Indeed,
by applying \eqref{111} with $f(x,y)=w_{y}(x,y)$ and \eqref{Rakotoson derivation formula},
we finally get
\begin{align*}
\frac{\partial^{2}}{\partial y^{2}}\int_{0}^{r}w^{\circledast}(\sigma,y)\,d\sigma
 &=\frac{\partial}{\partial y}\int_{\left\{x:\,w(x,y)>w^{\circledast}(r,y)\right\}}\frac{\partial}{\partial y}w(x,y)\,d\gamma(x) \\
& =\int_{\left\{x:\,w(x,y)>w^{\circledast}(r,y)\right\}}\frac{\partial^{2}}{\partial y^{2}}w(x,y)\,d\gamma(x) \\
&\quad+\int_{\left\{x:\,w(x,y)=w^{\circledast}(r,y)\right\}}\frac{\big(\frac{\partial}{\partial
y}w(x,y)\big)^{2}}{|\nabla_{x}w|}\varphi(x)\,d\mathcal{H}^{n-1}(x)\\
&\quad-\frac{
 \bigg(\displaystyle\int_{\left\{x:\,w(x,y)=w^{\circledast}(r,y)\right\}}\frac{\frac{\partial}{\partial
y}w(x,y)}{\|\nabla_{x}w|}\varphi(x)\,d\mathcal{H}^{n-1}(x)\bigg)^{2}
}
{
 \bigg(\displaystyle\int_{\left\{x:\,w(x,y)=w^{\circledast}(s,y)\right\}}\frac{1}{\left\vert \nabla
_{x}w\right\vert }\varphi(x)\,d\mathcal{H}^{n-1}(x)\bigg)
}.
\end{align*}
\end{proof}

\begin{remark}
The sum of the last two terms to the right-hand side of \eqref{Der2} is nonpositive, see \cite[Remark 2.8]{AlDiaz}.
\end{remark}

The following Lemma shows that we can actually apply the second order derivation formula
\eqref{Der2} to the solution $w$ to the extension problem \eqref{Problema estensione}, namely,
when $w=\mathcal{P}^s_yu$ is the extension of the solution
$u\in \mathcal{H}^s(\Omega,\gamma)$ to the linear problem \eqref{Problema}.

\begin{lemma}\label{lemma regolarita}
If $f\in L^{2}(\Omega,\gamma)$ then the second order derivation formula \eqref{Der2} can be applied to the solution $w$ to
problem \eqref{Problema estensione}.
\end{lemma}

\begin{proof}
Since $w\in C^\infty(\mathcal{C}_{\Omega})$,
by classical results on solutions of elliptic equations with analytic
coefficients (see for instance \cite{Hashimoto}),
$w$ is analytic. Hence condition \eqref{cond aggiuntiva} holds.
Next we have to show that the functions $w(x,y)\varphi
(x)$ and $\partial_{y}w(x,y)\varphi(x)$ are Lipschitz in $y\in(\varepsilon,T)$,
uniformly with respect to $x\in \Omega$. This follows because it is known that the solution
to the extension problem has the regularity $w\in  C^\infty((0,\infty);\mathcal{H}^s(\Omega,\gamma))$,
see \cite{Gale-Miana-Stinga, Stinga-Torrea}.
For the sake of completeness, we also give a direct proof of this regularity result.
By Theorem \ref{existweaksollin} and using the well known identity
$\frac{d}{dt}\big(t^\nu\mathcal{K}_\nu(t)\big)=-t^\nu\mathcal{K}_{\nu-1}(t)$, for $\nu\in\R$, it follows that
\[
\partial_{y}w=-C_s\sum_{k=1}^{\infty}(\lambda_k^{1/2}y)^s\mathcal{K}_{s-1}(\lambda_k^{1/2}y)
\frac{\langle f,\psi_{k}\rangle_{L^{2}(\Omega,\gamma)}}{\lambda_k^{s-1/2}}\psi_{k}(x)
\]
and
\[
\partial_{yy}w=-C_s\sum_{k=1}^{\infty}\big[(\lambda_k^{1/2}y)^{s-1}\mathcal{K}_{s-1}(\lambda_k^{1/2}y)
-(\lambda_k^{1/2}y)^{s-1}\mathcal{K}_{s-2}(\lambda_k^{1/2}y)y\big]
\frac{\langle f,\psi_{k}\rangle_{L^{2}(\Omega,\gamma)}}{\lambda_k^{s-1}}\psi_{k}(x).
\]
Then, as $\mathcal{K}_\nu(t)\sim\sqrt{\frac{\pi}{2t}}e^{-t}$, as $t\to\infty$, and
$\mathcal{K}_\nu(t)\sim C_\nu t^{-\nu}$, as $t\to0$, we get
\begin{align*}
\int_{0}^{\infty}y^a\int_{\Omega}|\partial_{y}w|^{2}\,d\gamma(x)\,dy
&=C_s\sum_{k=1}^\infty\bigg[\int_0^\infty y^a|(\lambda_k^{1/2}y)^s\mathcal{K}_{s-1}(\lambda_k^{1/2}y)|^2\,dy\bigg]
\frac{|\langle f,\psi_{k}\rangle_{L^{2}(\Omega,\gamma)}|^2}{\lambda_k^{2s-1}} \\
&=C_s\sum_{k=1}^\infty\bigg[\int_0^\infty r|\mathcal{K}_{s-1}(r)|^2\,dr\bigg]
\frac{|\langle f,\psi_{k}\rangle_{L^{2}(\Omega,\gamma)}|^2}{\lambda_k^{s}}
\leq \frac{C_s}{\lambda_{1}^{s}}\|f\|^{2}_{L^{2}(\Omega,\gamma)},
\end{align*}
and $w(x,y)\varphi(x)$ is Lipschitz with respect to $y\in(0,\infty)$, uniformly in $x$.
On the other hand,
\begin{align*}
\int_{\e}^{\infty}&y^a\int_{\Omega}|\partial_{yy}w|^{2}\,d\gamma(x)\,dy \\
&\leq C_{s,\varepsilon}\sum_{k=1}^{\infty}\int_\e^\infty y^a|(\lambda_k^{1/2}y)^{s-1}[(\lambda_k^{1/2}y)^{-1/2}
e^{-\lambda_k^{1/2}y}(1+y)]|^2\,dy
\frac{|\langle f,\psi_{k}\rangle_{L^{2}(\Omega,\gamma)}|^2}{\lambda_k^{2s-2}}\\
&\leq C_{s,\varepsilon}\sum_{k=1}^{\infty}\int_\e^\infty y^{-2}e^{-2\lambda_k^{1/2}y}(1+y)^2\,dy
\frac{|\langle f,\psi_{k}\rangle_{L^{2}(\Omega,\gamma)}|^2}{\lambda_k^{s-1/2}}\\
&\leq C_{s,\varepsilon}\sum_{k=1}^{\infty}\int_\e^\infty e^{-2\lambda_k^{1/2}y}\,dy
\frac{|\langle f,\psi_{k}\rangle_{L^{2}(\Omega,\gamma)}|^2}{\lambda_k^{s-1/2}}\\
&= C_{s,\varepsilon}\sum_{k=1}^{\infty}\frac{e^{-2\e\lambda_{k}^{1/2}}}{\lambda_k^s}
|\langle f,\psi_{k}\rangle_{L^{2}(\Omega,\gamma)}|^{2}\leq
\frac{C_{s,\varepsilon}}{\lambda_1^s}\|f\|^{2}_{L^{2}(\Omega,\gamma)}.
\end{align*}
Hence $\partial_{y}w(x,y)\varphi(x)$ is Lipschitz
with respect to $y\in(\e,\infty)$, uniformly in $x\in\Omega$.
\end{proof}

\section{The comparison result}\label{Section:comparison}

With the previous results at hand, we are now in position to prove the main result of the paper.

\begin{theorem}[Comparison result]\label{primoteoremadiconfronto}
Let $\Omega$ be an open subset of $\mathbb{R}^{n}$ with $\gamma(\Omega)<1$.
Let $u$ and $\psi$ be the weak solutions to \eqref{Problema} and \eqref{problema simm}, respectively,
with $f\in L^{2}(\Omega,\gamma)$. Then
\begin{equation}\int_{0}^{r}u^{\circledast}(\sigma)\,d\sigma\leq\int_{0}^{r}\psi^{\circledast}(\sigma)\,d\sigma,\label{confronto}
\quad\hbox{for all}~0\leq r\leq \gamma(\Omega),\end{equation}
that is,
\[
u^{\displaystyle\star}\prec  \psi.
\]
\end{theorem}

\begin{proof}
By making the change of variables $y=(2s)z^{1/(2s)}$ (see \cite{Caffarelli-Silvestre}), we can write the extension problems
\eqref{Problema estensione} and \eqref{Problema estensione sim} as
\begin{equation}\label{Problema estensione z}
\begin{cases}
-\mathcal{L}w+z^{2-1/s}w_{zz}=0,&\hbox{in}~\mathcal{C}_{\Omega},\\
w=0,&\hbox{on}~\partial_{L}\mathcal{C}_{\Omega},\\
-\underset{z\rightarrow0^{+}}{\lim}w_z=d_sf,&\hbox{on}~\Omega.
\end{cases}
\end{equation}
and
\begin{equation}\label{Problema estensione sim z}
\begin{cases}
-\mathcal{L}v+z^{2-1/s}v_{zz}=0,&\hbox{in}~\mathcal{C}_{\Omega}^{\displaystyle\star},\\
v=0,&\hbox{on}~\partial_{L}\mathcal{C}_{\Omega}^{\displaystyle\star},\\
-\underset{z\rightarrow0^{+}}{\lim}v_z=d_sf^{\displaystyle\star},&\hbox{on}~\Omega^{\displaystyle\star},
\end{cases}
\end{equation}
for some explicit constant $d_s>0$, respectively.
Now, since $u$ is the trace on $\Omega$ of the
solution $w$ to \eqref{Problema estensione z} and $\psi$ is the trace on $\Omega^{\bigstar}$ of the solution $v$ to
\eqref{Problema estensione sim z}, the result will immediately
follow once we prove the concentration comparison inequality
\begin{equation}
\int_{0}^{r}w^{\circledast}(\sigma,z)\,d\sigma\leq\int_{0}^{r}v^{\circledast}(\sigma,z)\,d\sigma,
\quad\hbox{for all}~0\leq r\leq\gamma(\Omega),~\hbox{for any fixed}~z\geq0.
\label{concentrazione}%
\end{equation}
We recall that $w$ is smooth for any $z>0$. For a fixed
$z>0$ and $t>0$, let
\[\varsigma_{h}^{z}(x):=
\begin{cases}
\mathrm{sign}\,w(x,z),&\hbox{if}~|w(x,z)|\geq t+h,\\ \medskip
\dfrac{|w(x,z)|-t}{h}\,\mathrm{sign}\,w,&\hbox{if}~t<|w(x,z)|<t+h,\\
0,&\hbox{otherwise}.
\end{cases}\]
By multiplying the first equation in \eqref{Problema estensione z} by $\varsigma_{h}^{z}(x)$
and integrating over $\Omega$ with respect to the Gaussian measure, we obtain
\begin{align*}
\frac{1}{h}\int_{\left\{x:\,t<|w(x,z)|<t+h\right\}}|\nabla_{x}w|^{2}\,d\gamma&-\frac{z^{2-1/s}}{h}
\int_{\left\{x:\,|w(x,z)|>t+h\right\}}\frac{\partial^{2}w}{\partial z^{2}}\,d\gamma\\
&-\frac{z^{2-1/s}}{h}\int_{\left\{x:\,t<|w(x,z)|<t+h\right\}}\frac{\partial^{2}w}{\partial
z^{2}}(|w|-t)\,\mathrm{sign}\,w\,d\gamma=0.
\end{align*}
Letting $h\rightarrow0$ we obtain%
\begin{equation}
-\frac{\partial}{\partial t}\int_{\left\{x:\,\left\vert w(x,z)\right\vert >t\right\}}\left\vert \nabla_{x}%
w\right\vert ^{2}\,d\gamma(x)-z^{2-1/s}\int_{\left\{x:\,\left\vert w(x,z)\right\vert >t\right\}}\frac
{\partial^{2}w}{\partial z^{2}}\,d\gamma(x)=0. \label{sostituendo}%
\end{equation}
\noindent On the other hand, the coarea formula \eqref{coarea} and the isoperimetric
inequality with respect to the Gaussian measure \eqref{dis isop} give
\begin{equation}
-\frac{\partial}{\partial t}\int_{\left\{x:\,\left\vert w(x,z)\right\vert >t\right\}}\!\left\vert \nabla_x
w\right\vert \, d\gamma(x)\geq\int_{\partial\left\{  x:\,\left\vert
w(x,z)\right\vert >t\right\}^{\star}}\varphi(x)\,d\mathcal{H}^{n-1}%
(x)=\frac{1}{\sqrt{2\pi}}\exp\Bigg(-\frac{\big[\Phi^{-1}\big(
\gamma_{w}(t)\big)\big]^{2}}{2}\Bigg)  ,
\label{isope}%
\end{equation}
where $\left\{x:\,  \left\vert w(x,z)\right\vert >t\right\}  ^{\displaystyle\star}$ is the
half-space having Gauss measure $\gamma_{w}(t)$.
By H\"{o}lder's inequality,
$$\frac{1}{h}\int_{\left\{x:\,t<|w(x,z)|<t+h\right\}}|\nabla_x w|\,d\gamma(x)\leq\bigg(\frac{1}{h}
\int_{\left\{x:\,t<|w(x,z)|<t+h\right\}}|\nabla_x w|^2\,d\gamma(x)\bigg)^{1/2}\bigg(\frac{1}{h}\int_{\left\{x:\,t<|w(x,z)|<t+h\right\}}\,d\gamma(x)\bigg)^{1/2},$$
for any $h>0$. Hence, by taking $h\to0$,
\[
-\frac{\partial}{\partial t}\int_{\left\{x:\,\left\vert w(x,z)\right\vert >t\right\}}\left\vert \nabla_{x}%
w\right\vert \ d\gamma(x)\leq\left(  -\frac{\partial}{\partial t}\int_{\left\{x:\,\left\vert w(x,z)\right\vert >t\right\}}
\left\vert \nabla_{x}w\right\vert ^{2}\, d\gamma(x)\right)
^{1/2}\left(  -\frac{\partial}{\partial t}\int_{\left\{x:\,\left\vert w(x,z)\right\vert >t\right\}}%
\, d\gamma(x)\right)  ^{1/2}.
\]
Then \eqref{isope} yields
\begin{equation}
-\frac{\partial}{\partial t}\int_{\left\{x:\,\left\vert w(x,z)\right\vert >t\right\}}\left\vert \nabla_{x}%
w\right\vert ^{2}\, d\gamma(x)\geq\frac{1}{2\pi}\left(  -\gamma_{w}^{\prime
}(t)\right)  ^{-1}\exp\!\left(  -\left[  \Phi^{-1}\big(\gamma_{w}(t)\big)  \right]  ^{2}\right)  . \label{isop+holder}%
\end{equation}
By plugging \eqref{isop+holder} into \eqref{sostituendo} we have
$$-z^{2-1/s}\int_{\left\{x:\,\left\vert w(x,z)\right\vert >t\right\}}\frac{\partial^{2}w}{\partial z^{2}%
}\ d\gamma(x)-\!\frac{1}{2\pi}\left(  \gamma_{w}^{\prime}(t)\right)
^{-1}\exp\left(  -\left[  \Phi^{-1}\big(\gamma_{w}(t)
\big)\right]  ^{2}\right)  \leq0.$$
Now we set
\[
W(r,y):=\int_{0}^{r}w^{\circledast}(\sigma,z)\,d\sigma.
\]
Using Lemma \ref{lemma regolarita} and the second order derivation formula (\ref{Der2})
we
find that $W$ verifies the following differential inequality%
\begin{equation}
-z^{2-1/s}\frac{\partial^{2}W}{\partial z^{2}}-p(r)\frac{\partial^{2}W}{\partial r^{2}%
}\leq0 \label{eq W}%
\end{equation}
for a.e.~$(r,z)\in(0,\gamma(\Omega))\times(0,\infty)$, where $p(r)=\frac{1}{2\pi}\exp(-[\Phi^{-1}(r)]^{2})$.
Moreover, the first order derivation formula
\eqref{Rakotoson derivation formula} implies%
\[
\frac{\partial W}{\partial z}(r,z)=\frac{\partial}{\partial z}\int
_{\{x:\,w(x,z)>w^{\circledast}(r,z)\}}w(x,z)\,d\gamma(x)=\int_{\{x:\,w(x,z)>w^{\circledast
}(r,z)\}}\frac{\partial}{\partial z}w(x,z)\,d\gamma(x).
\]
Then, by the Hardy--Littlewood inequality \eqref{Hardy-Litt}, we easily infer
\begin{align*}
\frac{\partial W}{\partial z}(r,0)  &  =\int_{\{x:\,w(x,0)>w^{\circledast}(r,0)\}}%
\frac{\partial w}{\partial z}(x,0)\, d\gamma(x)=-d_s\int_{\{x:\,u(x)>u^{\circledast}%
(r)\}}f(x)\, d\gamma(x)\\
&  \geq-d_s\int_{0}^{r}f^{\circledast}(\sigma)\,d\sigma,\quad\hbox{for}~r\in(0,\gamma(\Omega)).
\end{align*}
Therefore $W$ satisfies the following boundary conditions%
\begin{align*}
W(0,z)&  =0,\text{ \ \ } z\in\left[  0,\infty\right), \\
\frac{\partial W}{\partial r}(\gamma(\Omega),z)  &  =0,\text{ \ \ }
z\in\left[  0,\infty\right), \\
\frac{\partial W}{\partial z}(r,0)  &  \geq-d_s\int_{0}^{r}f^{\circledast}%
(\sigma)\,d\sigma,\text{ \ for }r\in(0,\gamma(\Omega)).
\end{align*}
Next let us turn our attention to problem \eqref{Problema estensione sim}. By Proposition \ref{prop chiacchio},
it follows that the function
$\displaystyle\eta(x,t):=\big(e^{-t(\mathcal{L}_{\Omega^{\displaystyle\star}})}f^{\displaystyle\star}\big)(x)$,
is rearranged with respect to $x$, that is, $\eta(x,t)=\eta^{\displaystyle\star}(x,t)$. Recall the semigroup formula \eqref{eq:semigroup formula f}:
\[
v(x,y)=\frac{1}{\Gamma(s)}\int_{0}^{\infty}e^{-y^{2}/(4t)}
\eta(x,t)\,\frac{dt}{t^{1-s}}.
\]
It is then clear that (even after the change of variables $y=(2s)z^{1/(2s)}$)
$v$ is rearranged with respect to $x$ as well,
namely, $v(x,z)=v^{\displaystyle\star}(x,z)$. This implies that the level sets of
$v(\cdot,z)$ are half-spaces, which gives in turn that all the inequalities involved
in the symmetrization arguments for the solution $u$ we performed above
become equalities for $v$. Therefore, if
$$V(r,z):=\int_{0}^{r}v^{\circledast}(\sigma,z)\,d\sigma,$$
then
\begin{equation}
-z^{2-1/s}\frac{\partial^{2}V}{\partial z^{2}}-p(r)\frac{\partial^{2}%
V}{\partial r^{2}}=0. \label{eq V}
\end{equation}
Regarding the boundary conditions, we have
\begin{align*}
\frac{\partial V}{\partial z}(r,0)  &  =-d_s\int_{\{x:\,\psi(x_{1})>\psi^{\circledast}(r)\}}f^{\bigstar}(x)\, d\gamma(x)\\
&  =-d_s\int_{\Phi^{-1}(r)}^{\infty}f^{\circledast}(\Phi^{-1}(x_{1}))\, d\gamma(x)\\
&  =-d_s\int_{0}^{r}f^{\circledast}(\sigma)\,d\sigma,\quad\hbox{for}~r\in(0,\gamma(\Omega)).
\end{align*}
Then $V$ satisfies:%
\begin{align*}
V(0,z)  &  =0,\text{ \ \ } z\in[0,\infty), \\
\frac{\partial V}{\partial r}(\gamma(\Omega),z)  &  =0,\text{ \ \ }
z\in[0,\infty), \\
\frac{\partial V}{\partial z}(r,0)  &  =-d_s\int_{0}^{r}f^{\circledast}%
(\sigma)\,d\sigma,\text{ \ for }r\in(0,\gamma(\Omega)).
\end{align*}
If we put $\displaystyle Z(r,z):=W(r,z)-V(r,z)=\int_{0}^{r}[w^{\circledast}(\sigma,z)-v^{\circledast}(\sigma,z)]\,d\sigma$,
then \eqref{eq W} and \eqref{eq V} imply that $Z$ is a subsolution to
\begin{equation}
-z^{2-1/s}\frac{\partial^{2}Z}{\partial z^{2}}-p(r)\frac{\partial^{2}%
Z}{\partial r^{2}}\leq0,\label{subsolution}
\end{equation}
for a.e.~$(r,z)\in(0,\gamma(\Omega))\times(0,\infty)$,
together with the following boundary conditions
\begin{align}
Z(0,z)  &  =0,\text{ \ \ }z\in[0,\infty),\nonumber \\
\frac{\partial Z}{\partial r}(\gamma(\Omega),z)  &  =0,\text{ \ \ }\label{boundcond}
z\in[0,\infty), \\
\frac{\partial Z}{\partial z}(r,0)  &  \geq0,\text{ \ for }r\in(0,\gamma(\Omega)).\nonumber
\end{align}
Moreover, since $\|w(\cdot,z)\|_{L^{2}(\Omega,\gamma)},\|v(\cdot,z)\|_{L^{2}(\Omega^{\displaystyle\star},\gamma)}
\rightarrow0$, as $z\rightarrow\infty$, we have
$Z(r,z)\rightarrow 0$, as $z\rightarrow\infty$, uniformly in $r$. Now we
claim that $Z\leq 0$ in $[0,\gamma(\Omega)]\times[0,\infty)$.
Indeed, observe that \eqref{subsolution} can be rewritten as
\[
-p(r)^{-1}\frac{\partial^2Z}{\partial z^2}-z^{-2+1/s}
\frac{\partial^2Z}{\partial r^2}\leq0.
\]
Therefore, by multiplying both sides by $Z_{+}$, the positive part of $Z$, and integrating by parts over the strip $(0,\gamma(\Omega))\times(0,\infty)$, the boundary conditions
\eqref{boundcond}
and the fact that $Z(r,z)\rightarrow0$ as $z\rightarrow\infty$ imply
\begin{multline*}
\int_{0}^{\gamma(\Omega)}p(r)^{-1}\frac{\partial Z}{\partial z}(r,0)Z_{+}(r,0)\,dr
+\int_{0}^{\infty}\int_{0}^{\gamma(\Omega)}z^{-2+1/s}\bigg|\frac{\partial Z_{+}}{\partial r}\bigg|^{2}\,dr\,dz\\
+\int_{0}^{\infty}\int_{0}^{\gamma(\Omega)}
p(r)^{-1}\bigg|\frac{\partial Z_{+}}{\partial z}\bigg|^{2}\,dr\,dz\leq0,
\end{multline*}
namely,
\[
\int_{0}^{\infty}\int_{0}^{\gamma(\Omega)}z^{-2+1/s}\bigg|\frac{\partial Z_{+}}{\partial r}\bigg|^{2}\,dr\,dz
+\int_{0}^{\infty}\int_{0}^{\gamma(\Omega)}
p(r)^{-1}\bigg|\frac{\partial Z_{+}}{\partial z}\bigg|^{2}\,dr\,dz\leq0.
\]
Thus $Z_{+}\equiv0$ and the concentration comparison inequality \eqref{concentrazione} follows.
\end{proof}

\section{Regularity estimates}\label{Section:regularity}

We first introduce the Zygmund spaces, which appear naturally in the regularity scale for solutions to elliptic
equations with Gaussian measure in the local setting, see \cite{dFP}. We refer the reader to the monograph \cite{Bennett} for details about all the related properties we will use for our purposes.

\begin{definition}[Zygmund spaces]
Let $1\leq p<\infty$ and $\alpha\in\R$. The Zygmund space $L^p(\log L)^\alpha(\Omega,\gamma)$ is defined as the space of all measurable functions $u:\Omega\to\R$ such that
$$\int_\Omega\big[|u(x)|\log^\alpha(2+|u(x)|)\big]^p\,d\gamma(x)<\infty.$$
\end{definition}
If $\alpha=0$ the Zygmund space $L^p(\log L)^0(\Omega,\gamma)$  coincides with the weighted space $L^p(\Omega,\gamma)$. Moreover, if $p>q$ and
$\alpha,\beta\in\R$ then
$$L^p(\log L)^\alpha(\Omega,\gamma)\subset L^q(\log L)^\beta(\Omega,\gamma).$$
When $p=q$ and $\alpha >\beta$ one can prove that
\begin{equation}
L^p(\log L)^\alpha(\Omega,\gamma)\subset L^p(\log L)^\beta(\Omega,\gamma).\label{secondembed}
\end{equation}

\begin{remark}
The Zygmund space $L^p(\log L)^\alpha(\Omega,\gamma)$ can be equivalently defined as the space of measurable functions $u:\Omega\to\R$ such that the quantity (which is a quasi-norm in this space)
\begin{equation}
\bigg(\int_0^{\gamma(\Omega)}\big[(1-\log t)^\alpha
u^{\circledast}(t)\big]^p\,dt\bigg)^{1/p}\label{qnorm}
\end{equation}
is finite. Moreover, $L^p(\log L)^\alpha(\Omega,\gamma)$ is a Banach space when equipped with the norm
\begin{equation}
\|u\|_{L^p(\log L)^\alpha(\Omega,\gamma)}^p=\int_0^{\gamma(\Omega)}\big[(1-\log t)^\alpha
u^{\circledast\circledast}(t)\big]^p\,dt,\label{normavera}
\end{equation}
where
\[
u^{\circledast\circledast}(t):=\frac{1}{t}\int_{0}^{t}u^{\circledast}(\sigma)\,d\sigma.
\]
We stress that the quasi-norm \eqref{qnorm} is equivalent to the norm \eqref{normavera}, see \cite{Bennett} for more details.
\end{remark}

The main result of this section is the following regularity result for solutions to the fractional
nonlocal OU problem \eqref{Problema} in terms of the data $f$.
Notice that when $s=1$ we recover the corresponding regularity results
for the OU equation via Gaussian symmetrization contained in \cite{dFP}.

\begin{theorem}[Regularity estimates]\label{thm:integrability}
Let $\Omega$ be an open subset of $\mathbb{R}^{n}$, $n\geq2$, such that $\gamma(\Omega)\leq1/2$.
Fix $0<s<1$.
If $f\in L^{p}(\log L)^\alpha(\Omega,\gamma)$, where $\alpha\in\R$ for $2< p<\infty$, and $\alpha\geq-\frac{s}{2}$ for $p=2$, then the solution $u$ to
\eqref{Problema} belongs to
$L^{p}(\log L)^{\alpha+s}(\Omega,\gamma)$ and
$$\|u\|_{L^{p}(\log L)^{\alpha+s}(\Omega,\gamma)}\leq C\|f\|_{L^{p}(\log L)^\alpha(\Omega,\gamma)},$$
for a positive constant $C=C(n,p,\alpha,s,\gamma(\Omega))$ which is independent on $u$ and $f$.
\end{theorem}

In order to prove Theorem \ref{thm:integrability} we will first show that the space $\mathcal{H}^{s}(\Omega,\gamma)$ is embedded in the Zygmund space $L^2(\log L)^{s/2}(\Omega,\gamma)$. This will allow us to choose the datum
$f$ in the dual space $L^2(\log L)^{-s/2}(\Omega,\gamma)$ in problem \eqref{Problema}. In this way
Definition \ref{Defprobest}
will still make sense and $u=w(\cdot,0)$, where $w$ is the solution to
\eqref{Problema estensione}, will be the unique weak solution to problem
\eqref{Problema}. Towards this end we introduce the fractional Gaussian Sobolev space $H^{s}(\Omega,\gamma)$ as the real interpolation space defined by
\[
H^{s}(\Omega,\gamma)=\left[  H^{1}(\Omega,\gamma),L^{2}(\Omega
,\gamma)\right]  _{1-s}.
\]

\begin{lemma}
For any $u\in\mathcal{H}^{s}(\Omega,\gamma)$ the following inequality holds
\begin{equation}
\int_{0}^{\gamma(\Omega)}[(1-  \log r)^{s/2}u^{\circledast}(r)]^{2}\,dr\leq C\Vert u\Vert_{\mathcal{H}^{s}(\Omega
,\gamma)}^2 \label{emblorentz}%
\end{equation}
where $C$ is a positive constant depending on $n,s$ and $\Omega$. In particular,
\[
\mathcal{H}^{s}(\Omega,\gamma)\hookrightarrow L^{2}(\log L)^{s/2}(\Omega,\gamma).
\]
\end{lemma}

\begin{proof}
Given any function $u\in\mathcal{H}^{s}(\Omega,\gamma)$ we consider the
extension $\widetilde{u}$ of $u$ by zero outside of $\Omega$. Since $\widetilde
{u}\in H^{s}(\mathbb{R}^{n},\gamma)$ and this last space
coincides with the Gaussian Besov space $B^{s}(\mathbb{R}^{n},\gamma)$ (see
\cite{Nikitin}), the embedding result contained in \cite[Theorem 23]{MaMi} yields
\begin{equation*}
 \int_{0}^{1/2}[(1-  \log r)^{s/2}u^{\circledast}(r)]^{2}\,dr\leq
 C\Vert\widetilde{u}\Vert^{2}_{H^{s}(\mathbb{R}^{n},\gamma)},
\end{equation*}
for some constant $C>0$. A change of variable and the monotonicity of the decreasing rearrangement $u^{\circledast}$ lead to
\begin{equation}
\int_{0}^{1}[(1-  \log r)^{s/2}u^{\circledast}(r)]^{2}\,dr
\leq 2 \int_{0}^{1/2}[(1-  \log r)^{s/2}u^{\circledast}(r)]^{2}\,dr
\leq 2C\Vert\widetilde{u}\Vert^{2}_{H^{s}(\mathbb{R}^{n},\gamma)}.\label{emblorentz2}
\end{equation}
Now we observe that the Exact Interpolation Theorem (see \cite[Theorem 7.23]{AdamsFourn}) implies that
extending any function $u\in\mathcal{H}^{s}(\Omega,\gamma)$ by zero
outside of $\Omega$ defines a continuous extension map between $\mathcal{H}^{s}%
(\Omega,\gamma)$ and $H^{s}(\mathbb{R}^{n},\gamma)$. Thus it follows that the norm at the right-hand side of
\eqref{emblorentz2} is bounded (up to a constant depending on $n,s$ and $\Omega$) by $\Vert u\Vert_{\mathcal{H}^{s}(\Omega,\gamma)}^2$ and the result follows.
\end{proof}

With these results at hand, we are able to show the generalization of the comparison result (Theorem \ref{primoteoremadiconfronto}) for $f$ in Zygmund spaces.

\begin{corollary}
Assume that $f\in L^2(\log L)^{-s/2}(\Omega,\gamma)$. Then Theorem \ref{primoteoremadiconfronto} still holds.
\end{corollary}

\begin{proof}
Let $f_{n}$ be a sequence of smooth function such that $f_{n}\rightarrow f$ strongly in
$L^{2}(\log L)^{-s/2}(\Omega,\gamma)$. Let $w_{n}$ be the unique
weak solution to problem \eqref{Problema estensione} with data $f_{n}$. By choosing $w_{n}$ as a test function in \eqref{weakform} we have
\begin{align*}
\iint_{\mathcal{C}_{\Omega}} y^{a}|\nabla_{x,y}w_{n}|^{2}\,d\gamma(x)\,dy
&= c_{s}^{-1}\int_{\Omega}f_{n}(x)w_{n}(x,0)\,d\gamma(x) \\
&\leq c_{s}^{-1}\|w_{n}(x,0)\|_{L^{2}(\log
L)^{s/2}(\Omega,\gamma)}\|f_{n}(x,0)\|_{L^{2}(\log L)^{-s/2}(\Omega,\gamma)}.
\end{align*}
Next we use \eqref{emblorentz} and the trace inequality \eqref{trace} to find
\[
\iint_{\mathcal{C}_{\Omega}} y^{a} |\nabla_{x,y}w_{n}|^{2}\,d\gamma(x)\,dy\leq C\|w_{n}\|_{H_{0,L}^{1}(\mathcal{C}_{\Omega},d\gamma(x)\otimes
y^{a}dy)}\,\|f_{n}(x,0)\|_{L^{2}(\log L)^{-s/2}(\Omega,\gamma)}.
\]
This allows us to extract a subsequence from $\left\{w_{n}\right\}$ (still labeled by $\left\{w_{n}\right\}$), such that
$w_{n}\rightharpoonup w$ weakly in $H_{0,L}^{1}(\mathcal{C}_{\Omega},d\gamma(x)\otimes y^{a}dy)$.
Then the compact embedding established in Proposition \ref{compactness} gives that,
up to a new subsequence, $w_{n}(\cdot,0)\rightarrow w(\cdot,0)$ strongly in $L^{2}(\Omega,\gamma)$.
Thus we can pass to the limit in the weak formulation \eqref{weakform} of $w_{n}$ and find that $w$ solves problem \eqref{Problema estensione}
corresponding to the data $f$. Thus $u:=w(\cdot,0)$ is the weak solution to problem \eqref{Problema}. In order to obtain the concentration inequality
\eqref{confronto}, we just observe that $f_{n}^{\displaystyle\star}$ approximates $f^{\displaystyle\star}$ in $L^{2}(\Omega^{\displaystyle\star},\gamma)$.
Then, if $\left\{w_{n}\right\}$ and $\left\{v_{n}\right\}$ are sequences of approximating solutions converging to $w$ and $v$ respectively, passing to the limit in the integral inequality
\begin{equation*}
\int_{0}^{s}w_{n}^{\circledast}(\sigma,0)\,d\sigma\leq\int_{0}^{s}v_{n}^{\circledast
}(\sigma,0)\,d\sigma,
\end{equation*}
we immediately get \eqref{confronto}.
\end{proof}

For the proof of Theorem \ref{thm:integrability} we need two further preliminary results, interesting in their own right. The following is a regularity
result for solutions of problems of the type \eqref{Problema} with \emph{rearranged} data, posed on the half-space $H$.

\begin{theorem}[Estimates for half-space solutions]\label{Regularity theorem}
Let $H=\{x\in\R^n:x_{1}>0\}$. Suppose that $h(x)=h^{\bigstar}(x)$, for all $x\in H$.
If $h\in L^{p}(\log L)^\alpha(H,\gamma)$ with $\alpha\in\R$ for $2<p<\infty$, and $\alpha\geq-\frac{s}{2}$ for $p=2$,
then the weak solution $\psi$ to
\begin{equation}\label{PPP}
\begin{cases}
\mathcal{L}^s\psi=h,&\hbox{in}~H,\\
\psi=0,&\hbox{on}~\partial H,
\end{cases}
\end{equation}
belongs to $L^{p}(\log L)^{\alpha+s}(H,\gamma)$ and
$$\|\psi\|_{L^{p}(\log L)^{\alpha+s}(H,\gamma)}\leq C\|h\|_{L^{p}(\log L)^\alpha(H,\gamma)},$$
for some constant $C=C(n,p,\alpha,s)>0$, which is independent on $\psi$ and $h$.
\end{theorem}

\begin{proof}
By \eqref{eq:fractional integral} and \eqref{est1}--\eqref{eta} we can write
$$\psi(x)=\frac{1}{\Gamma(s)}\int_0^\infty e^{-t(\mathcal{L}_H)}h(x)\,\frac{dt}{t^{1-s}}=
\frac{1}{\Gamma(s)}\int_0^\infty e^{-t\mathcal{L}}\widetilde{h}(x)\,\frac{dt}{t^{1-s}}=\mathcal{L}^{-s}\widetilde{h}(x).$$
Then the estimate follows from \cite[Theorem~5.7]{KM}.
\end{proof}

The next Lemma is a useful comparison principle for solutions of problems of the form \eqref{Problema} with rearranged data, having as a ground domain an
half-space of Gaussian measure larger than $1/2$.

\begin{lemma}[Comparison of half-space solutions]\label{semispazi}
Let $H_{\omega}=\{x\in\mathbf{\mathbb{R}}^{n}:x_{1}>\omega\}$, for some $\omega>0$.
Let $h\in L^{p}(\log L)^\alpha(H,\gamma)$ be a nonnegative function such that $h(x)=h^{\bigstar}(x)$ and let $\psi$ be the weak
solution to
$$\begin{cases}
\mathcal{L}^{s}\psi=h,& \hbox{in}~H_{\omega},\\
\psi=0,&\hbox{on}~\partial H_{\omega}.
\end{cases}$$
Then
$$\psi(x) \leq \zeta(x), \quad\hbox{for a.e.}~x\in H_{\omega},$$
where $\zeta$ is the weak solution to \eqref{PPP} with datum $\overline{h}$,
where $\overline{h}$ denotes the zero extension of $h$ in $H\setminus H_{\omega}$.
\end{lemma}

\begin{proof}
The function
\[
F(x,t):=e^{-t(\mathcal{L}_{H})}\overline{h}(x)-e^{-t(\mathcal{L}_{H_{\omega}})}h(x),
\]
solves the initial boundary value problem
$$\begin{cases}
\partial_tF=\Delta F-x\cdot\nabla F,&\hbox{in}~H_\omega\times(0,\infty), \\
F(x,t)\geq0,&\hbox{on}~\partial H_{\omega}\times(0,\infty), \\
F(x,0)=0,&\hbox{on}~H_{\omega}.\\
\end{cases}$$
Thus, by a standard maximum principle argument, $F\geq0$ in $H_{\omega}\times[0,\infty)$.
In other words,
$$e^{-t(\mathcal{L}_{H})}\overline{h}\geq e^{-t(\mathcal{L}_{H_{\omega}})}h,\quad
\hbox{for all}~x\in H_{\omega},~t\geq0.$$
Therefore, if $v$ and $\overline{v}$ denote the extensions as in
\eqref{StingaTorreasemigr} of $\psi$ and $\zeta$, respectively, then
$$\overline{v}(x,y)\geq v(x,y),\quad\hbox{for all}~x\in H_{\omega},~y\geq0.$$
The result follows by taking $y=0$ in this last inequality.
\end{proof}

Now we are finally able to present the proof of the regularity estimate, namely, Theorem \ref{thm:integrability}.

\begin{proof}[Proof of Theorem \ref{thm:integrability}]
Let $u$ be the weak solution to \eqref{Problema} defined in an open set $\Omega$
such that $\gamma(\Omega)\leq 1/2$, with corresponding datum $f$. By Theorem
\ref{primoteoremadiconfronto}, $u$ is less concentrated than the solution $\psi$ to \eqref{problema simm} defined in the half-space with the
same Gauss measure as $\Omega$ and datum $f^{\displaystyle\star}$. If $\gamma(\Omega)=1/2$ the assertion follows by Theorem \ref{Regularity theorem}. If
$\gamma(\Omega)<1/2$,
we first apply Lemma \ref{semispazi} to estimate $\psi$ in terms of the solution $\zeta$ to \eqref{PPP} defined in the
half-space $H=\{x\in\mathbf{\mathbb{R}}^{n}:x_{1}>0\}$ and having the extension of
$f^{\displaystyle\star}$ by zero to $H$ at the right-hand side.. Then Theorem \ref{Regularity theorem} allows us to conclude.
\end{proof}

\begin{remark}
We remark that other regularity results for problems involving fractional operators with bounded lower order terms,
but posed on bounded smooth domains, are contained in \cite{Grubb}.
\end{remark}

\section{Appendix: A semigroup method proof of the $L^p$ estimate}\label{Appendix}

For completeness and convenience of the reader, we give an alternative and more explicit proof of Theorem \ref{Regularity theorem} with $L^p$ data using the Mehler kernel to represent the inverse of the fractional
OU operator. Observe that such result is a particular case of Theorem \ref{Regularity theorem} since, when $f\in L^{p}(\Omega,\gamma)$, Theorem
\ref{Regularity theorem} and the embedding \eqref{secondembed} give $u\in L^{p}(\log L)^{s}(\Omega,\gamma)\subset L^{p}(\Omega,\gamma)$.

\begin{theorem}[Estimates for half-space solutions with $L^{p}$ data]\label{Regularity L^{p} theorem}
Let $H=\{x\in\R^n:x_{1}>0\}$. Suppose that $h(x)=h^{\bigstar}(x)$, for all $x\in H$. If $h\in L^{p}(H,\gamma)$, for $2\leq p<\infty$,
then the weak solution $\psi$ to \eqref{PPP} belongs to $L^{p}(H,\gamma)$ and
$$\|\psi\| _{L^{p}(H,\gamma)}\leq C\|h\|_{L^{p}(H,\gamma)},$$
for some constant $C=C(n,p,s)>0$, which is independent on $\psi$ and $h$.
\end{theorem}

\begin{proof}
The proof will be split in four steps.
\newline
\noindent\textbf{Step 1. The explicit solution via the semigroup kernel.}
By \eqref{eq:fractional integral}, and by using an abuse of notation, the solution $\psi$ to \eqref{PPP} can be written as
$$\psi(x)=\psi(x_{1})=\frac{1}{\Gamma(s)}
\int_0^\infty e^{-t(\mathcal{L}_H)}h(x)\,\frac{dt}{t^{1-s}}
=\int_{0}^{\infty}G(x_{1},y_{1})h(y_{1})\,d\gamma(y_{1}),$$
where (see \eqref{eq:eta})
$$G(x_{1},y_{1})=\frac{1}{\Gamma(s)}\int_{0}^{\infty}[M_{t}(x_{1},y_{1})-M_{t}(x_{1},-y_{1})]\,\frac{dt}{t^{1-s}}.$$
Next we write
\begin{equation}\label{G}
\begin{aligned}
G(x_{1},y_{1})&=\int_{0}^{c(p)}\cdots\,dt+\int_{c(p)}^{T(x_{1},y_{1})}\cdots\,dt+\int_{T(x_{1},y_{1})}^{\infty}\cdots\,dt \\
&=:G_{1}(x_{1},y_{1})+G_{2}(x_{1},y_{1})+G_{3}(x_{1},y_{1}),
\end{aligned}
\end{equation}
with $c(p)>1$ a suitable constant, and $T(x_{1},y_{1})=\max\{c(p),\log\left(
x_{1}^{2}+y_{1}^{2}\right)  \}.$
It follows that
\begin{equation}\label{aaaa}
\|\psi\|_{L^{p}(H,\gamma)}^{p}
\leq\sum_{j=1}^3\int_{0}^{\infty}\left(\int_{0}^{\infty}G_{j}(x_{1},y_{1})h(y_{1})\,d\gamma(y_{1})\right)^{p}d\gamma(x_{1}).
\end{equation}
\noindent\textbf{Step 2. Estimate of the term $j=1$ in \eqref{aaaa}.} We observe that by
\eqref{Stima Lp semigruppo} and \eqref{norma prolungamento} we get
\begin{equation}\label{ssm}
  \left\Vert \int_{-\infty}^{\infty}M_{t}(x_{1},y_{1})\widetilde
{h}(y_{1})\,d\gamma(y_{1})\right\Vert _{L^{p}(\R,\gamma)}
  \leq\|\widetilde{h}\|_{L^{p}(\mathbb{R},\gamma)}
=2\left\Vert h\right\Vert _{L^{p}(H,\gamma)},
\end{equation}
where $\widetilde{h}$ is defined like in \eqref{est1}. Tonelli's theorem,
Minkowski's inequality and (\ref{ssm}) yield%
\begin{align*}
&  \left\Vert \int_{0}^{\infty}G_{1}(x_{1},y_{1})h(y_{1})\,d\gamma(y_{1})\right\Vert _{L^{p}(H,\gamma)}\\
&  \leq c_s\int_{0}^{c(p)}\left\Vert \int_{0}^{\infty}\left[
M_{t}(x_{1},y_{1})-M_{t}(x_{1},-y_{1})\right]
h(y_{1})\,d\gamma(y_{1})\right\Vert _{L^{p}(H,\gamma)}\frac{dt}{t^{1-s}}\\
&=c_s\int_{0}^{c(p)}\left\Vert \int_{-\infty}^{\infty}M_{t}(x_{1},y_{1})\widetilde{h}(y_{1})\,d\gamma(y_{1})\right\Vert
_{L^{p}(\R,\gamma)}\frac{dt}{t^{1-s}} \\
&\leq 2c_s\left\Vert h\right\Vert _{L^{p}(H,\gamma)}\int_{0}^{c(p)}\frac{dt}{t^{1-s}}
=c_{s,p}\|h\|_{L^p(H,\gamma)}.
\end{align*}
\noindent\textbf{Step 3. Estimate of $G_{2}$ and $G_{3}$.} We prove that
$$\int_0^\infty\left(\int_0^\infty G_{j}^{p'}(x_{1},y_{1})\,d\gamma(y_{1})\right)^{p/p'}
d\gamma(x_{1})<\infty,\quad\hbox{for}~j=2,3.$$
By Jensen's inequality, it is enough to show that $G_{j}\in L^{p}(H\times H,\gamma\otimes\gamma)$, for $j=2,3$.
If $t>c(p)>1$ then $(1-e^{-2t})\sim1$ and $\left\vert M_{t}(x_{1},y_{1})\right\vert\leq
c\exp\left(  4e^{-t}\left\vert x_{1}\right\vert\left\vert y_{1}\right\vert \right)$, see \eqref{eq:Mehler kernel}.
It follows that
\begin{align*}
\left\vert G_{2}(x_{1},y_{1})\right\vert
&\leq c_s\int_{c(p)}^{T(x_{1},y_{1})}\vert M_{t}(x_{1},y_{1})-M_{t}(x_{1},-y_{1})\vert\,\frac{dt}{t^{1-s}} \\
& \leq \frac{c_s}{c(p)^{1-s}}\int_{c(p)}^{T(x_{1},y_{1})}\exp(4e^{-t}|x_{1}||y_{1}|)\,dt \\
& \leq \frac{c_s}{c(p)^{1-s}}\int_{c(p)}^{T(x_{1},y_{1})}\exp\left[2e^{-c(p)}(x_{1}^{2}+y_{1}^{2})\right]\,dt \\
& \leq\frac{c_s}{c(p)^{1-s}}\cdot\frac{T(x_{1},y_{1})}{\left(  \varphi(x_{1})\right)  ^{4e^{-c(p)}}
\left(\varphi(y_{1})\right)  ^{4e^{-c(p)}}}=:\widetilde{G}_{2}(x_{1},y_{1}).
\end{align*}
We then get $\widetilde{G}_{2}(x_{1},y_{1})\in L^{p}(H\times H,\gamma\otimes\gamma)$
if we choose $4pe^{-c(p)}<1$, that is, if $c(p)>\max\{1,\log(4p)\}$. Moreover, by Taylor's formula
and using that $t>1$ and $e^{-t}(|x_1|^2+|y_1|^2)<1$,
\begin{align*}
M_{t}(x_{1},y_{1})-M_{t}(x_{1},-y_{1}) &\leq C_n
\Big|\exp\Big(\frac{e^{-t}\langle x_1,y_1\rangle}{1-e^{-2t}}\Big)-\exp\Big(-\frac{e^{-t}\langle x_1,y_1\rangle}{1-e^{-2t}}\Big)\Big| \\
&\leq Ce^{-t}|\langle x_1,y_1\rangle|\exp(ce^{-t}\langle x_1,y_1\rangle) \\
&\leq Ce^{-t}(|x_1|^2+|y_1|^2)\exp(ce^{-t}(|x_1|^2+|y_1|^2)) \\
&\leq Ce^{-t}(|x_1|^2+|y_1|^2).
\end{align*}
Then
\begin{align*}
\left\vert G_{3}(x_{1},y_{1})\right\vert  &  \leq C_s\int_{T(x_{1},y_{1})}^{\infty}
\left\vert M_{t}(x_{1},y_{1})-M_{t}(x_{1},-y_{1})\right\vert\,\frac{dt}{t^{1-s}}
\leq C_{n,s}\big(|x_1|^2+|y_1|^2\big)\int_{T(x_{1},y_{1})}^{+\infty}e^{-t}dt\\
&  =C_{n,s}\big(\left\vert x_{1}\right\vert ^{2}+\left\vert y_{1}\right\vert^{2}\big)
e^{-T(x_{1},y_{1})}\leq C_{n,s}\in L^{p}(H \times H,\gamma\otimes\gamma).
\end{align*}
\noindent\textbf{Step 4. Estimates of the terms $j=2,3$ in \eqref{aaaa}.}
By H\"{o}lder's inequality and the estimates of Step 3, we get
\begin{align*}
&\int_{0}^{\infty}\left(\int_{0}^{\infty}G_{j}(x_{1},y_{1})h(y_{1})\,d\gamma(y_{1})\right)^{p}d\gamma(x_{1})\\
&  \leq\int_{0}^{\infty}\left(\int_{0}^{\infty}G_{j}^{p'}(x_{1},y_{1})\,d\gamma(y_{1})\right)^{p/p'}\left(
\int_{0}^{\infty}|h(y_{1})|^p\,d\gamma(y_{1})\right)d\gamma(x_{1})\leq
c\left\Vert h\right\Vert _{L^{p}(H,\gamma)}^{p},
\end{align*}
for $j=2,3$ and for some positive constant $c=c(n,p,s)$.\\
Hence the desired result follows by collecting Steps 2 and 4 in estimate \eqref{aaaa}.
\end{proof}

\bigskip
\noindent\textbf{Acknowledgements.} Research partially supported by GNAMPA of INdAM, ``Programma
triennale della Ricerca dell'Universit\`{a} degli Studi di Napoli "Parthenope"
- Sostegno alla ricerca individuale 2015-2017" (Italy) and
by Grant MTM2015-66157-C2-1-P form Government of Spain.



\end{document}